\documentclass[11pt]{article}       
\title{Multidimensional Hahn polynomials, intertwining functions on the symmetric group and Clebsch-Gordan coefficients }

\author{Fabio Scarabotti}

\usepackage{latexsym,amsbsy,amsmath,amscd,amssymb}  

\usepackage{epic,amsthm}       
\usepackage{amssymb,amsmath,latexsym}
 \newtheorem{definition}{Definition} [section]       
 \newtheorem{remark}[definition]{Remark}       
 \newtheorem{example}[definition]{Example}

 \newtheorem{proposition}[definition]{Proposition}       
 \newtheorem{theorem}[definition]{Theorem}       
 \newtheorem{corollary}[definition]{Corollary}       
  \newtheorem{lemma}[definition]{Lemma}

\begin{document}

\maketitle

\begin{abstract}  
We generalize a construction of Dunkl, obtaining a wide class intertwining functions on the symmetric group and a related family of multidimensional Hahn polynomials. Following a suggestion of Vilenkin and Klymik, we develop a tree-method approach for those intertwining functions. We also give a group theoretic proof of the relation between Hahn polynomials and Clebesh-Gordan coefficients, given analytically by Koornwinder and by Nikiforov, Smorodinski\u{i} and Suslov. Such relation is also extended to the multidimensional case.
\footnote{{\it AMS 2002 Math. Subj. Class.}Primary: 33C80. Secondary: 20C30, 33C45, 33C50, 81R05}
\end{abstract}

\section{Introduction}

One of the most fruitful method to study special functions is to use representation theory and harmonic analysis. In particular, discrete orthogonal polynomials can be treated by mean of the representation theory of finite group, and a lot of work in this direction was made by Delsarte, Dunkl and Stanton; see \cite{St1} for a useful survey. In \cite{Du3} Dunkl developed the framework of intertwining functions and in \cite{Du2} he used it to study Hahn polynomials. Indeed, the $S_{n-m}\times S_m-S_{n-h}\times S_h$-intertwining functions on the symmeric group are naturally expressed in terms of these polynomials. In \cite{Du5} he studied the $S_a\times S_b\times S_c-S_{n-m}\times S_m$ intertwining functions on the symmetric group $S_n$ (where $n=a+b+c$), expressing them in terms of a class of two dimensional Hahn polynomials. Those polynomials had been introduced by Karlin and McGregor in \cite{K-M}.\\

 The aim of the present paper is to extend Dunkl's results and to study the $S_{a_1}\times S_{a_2}\times\dotsb\times S_{a_h}-S_{n-m}\times S_m$-intertwining functions on $S_n$ ($a_1+a_2+\dotsb +a_h=n$). Following a suggestion of  Vilenkin and Klymik (see \cite{K-V}, p. 505), we develop the tree method for those intertwining functions and for the related Hahn polynomials. The key ingredient is an explicit Littlewood-Richardson rule for the irreducible representations of the symmetric group associated to two rows Young tableaux. The resulting Hahn polynomials are more general than those in \cite{K-M}, that correspond to a special choice of the tree. It is well known that Hahn polynomials are related to the Clebsch-Gordan coefficients for $SU(2)$; see \cite{Koornwinder,NikSus,NSU,SS}. We give a group theoretic approach to this relation and extend it to the multidimensional case. Indeed, the relation between Clebsch-Gordan coefficients for $SU(2)$ and the representation theory of the symmetric group was investigated in \cite{Jucys,Sullivan1,Sullivan2}, but in those references no connection is found with Dunkl's intertwining functions and Hahn polynomials (and with the analytical formulas in \cite{Koornwinder,NikSus,NSU,SS}). 
In view of these results, we do not compute connection coefficients between different bases of intertwining functions (this is the main problem in every tree method): they can be obtained by the equivalent theory on $SU(2)$.
 
The plan of the paper is the following. In section 2, we review the Hahn polynomials (using Dunkl's renormalized notation) and the Gel'fand pair $(S_n,S_{n-m}\times S_m)$ (also called the Johnson scheme). Then we give an explicit Littelwod.Richardson rule for the spherical representation of the Johnson scheme. This rule is expressed in terms of Hahn polynomials and Radon transforms and our formulas generalize those for the intertwining functions in \cite{Du2}. In section 3, we develop the tree method for the multidimensional Hahn polynomials and obtain a basis for the $S_{a_1}\times S_{a_2}\times\dotsb\times S_{a_h}-S_{n-m}\times S_m$-intertwining functions on the symmetric group. 
We also give a basis for the $S_{n-m}\times S_m-S_{a_1}\times S_{a_2}\times\dotsb\times S_{a_h}$-intertwining functions, that are more directly related to the Clebsch-Gordan coefficients. Though these functions may be obtained by the simple change of variable $g\rightarrow g^{-1}$, their study requires the use of the theory of induced representations. As a by-product, we give a group theoretic proof of the Regge's symmetries for the Hahn polynomials. In section four we introduce Clebsch-Gordan coefficients and show their relation with the intertwining functions on the symmetric group. We make use of the Schur-Weyl duality in the form developed by James in \cite{Ja2}, which is quite suitable for our purposes. First we analyzed the case of the tensor product of two irreducible $SU(2)$-representations, connecting it with the theory of $S_{n-m}\times S_m-S_{n-h}\times S_h$-intertwining functions on $S_n$. Then, in the last subsection, we show that the tree method for the intertwining functions on the symmetric group is equivalent to the tree method for the Wigner $3nj$-coefficients for $SU(2)$.

Another generalization of the results in \cite{Du5} is in \cite{ScarabottiSabc}, where an orthogonal basis for the $S_a\times S_b\times S_c$-invariant functions in the irreducible representation $S^{\alpha,\beta,\gamma}$ of $S_N$ (where $N=a+b+c=\alpha+\beta+\gamma$) is obtained.

%%%%%%%%%%%%%%%%%%%%%%%%%%%%%%%%%%%%%%%%%%%%%%%%%%%%%%%%%%%%%%%%%%%%%%%%%%%%%%%%%%
%%%%%%%%%%%%%%%%%%%%%%%%%%%%%%%%%%%%%%%%%%%%%%%%%%%%%%%%%%%%%%%%%%%%%%%%%%%%%%%%%%%%%
\section{Hahn polynomials and the Johnson scheme}
%%%%%%%%%%%%%%%%%%%%%%%%%%%%%%%%%%%%%%%%%%%%%%%%%%%%%%%%%%%%%%%%%%%%%%%%%%%%%%%%%%%%%

%%%%%%%%%%%%%%%%%%%%%%%%%%%%%%%%%%%%%%%%%%%%%%%%%%%%%%%%%%%%%%%%%%%%%%%%%%%%%%%%%%
%%%%%%%%%%%%%%%%%%%%%%%%%%%%%%%%%%%%%%%%%%%%%%%%%%%%%%%%%%%%%%%%%%%%%%%%%%%%%%%%%%%%%
\subsection{Hahn polynomials}\label{df}
%%%%%%%%%%%%%%%%%%%%%%%%%%%%%%%%%%%%%%%%%%%%%%%%%%%%%%%%%%%%%%%%%%%%%%%%%%%%%%%%%%%%%

We recall the basic properties of the one dimensional Hahn polynomials, using Dunkl's renormalized notation; \cite {Du2,Du5,ScarabottiSabc}.
For $m,a,b,c,x$ integers satisfying:

\begin{equation}\label{cab}
0\leq c \leq a+b,\quad 0\leq m\leq \min\{a,b,c,a+b-c\} \quad\text{ and }\quad \max\{c-b,0\}\leq x \leq\min\{a,c\},
\end{equation} 

they are given by the formula:

\[
E_m(a,b,c,x)=\sum_{j=\max\{0,x-c+m\}}^{\min\{m,x\}} 
(-1)^j\binom{m}{j}(b-m+1)_j(-x)_j(a-m+1)_{m-j}(x-c)_{m-j},
\]

where $(x)_k=x(x+1)(x+2)\dotsb(x+k-1)$.
Now we list some of their properties.

Transformation formula:

\begin{equation}\label{form7}
\sum_{x=\max\{0,c-d+y\}}^{\min\{c,y\}}\binom{y}{x}\binom{d-y}{c-x}E_m(a,b,c,x) = 
\binom{d-m}{c-m}E_m(a,b,d,y).
\end{equation}

Orthogonality relations:

\begin{equation}\label{orthrel}
\begin{split}
\sum_{x=\min\{0,c-b\}}^{\max\{a,c\}}&\binom{a}{x}\binom{b}{c-x}E_m(a,b,c,x)E_n(a,b,c,x)=\delta_{nm}\binom{a+b}{c}\binom{a+b}{m}^{-1}\times\\
&\times\frac{a+b-m+1}{a+b-2m+1}(a-m+1)_m(b-m+1)_m(c-m+1)_m(a+b-c-m+1)_m
\end{split}
\end{equation}

Particular values: for $c=m$ we have  

\begin{equation}\label{form9}
E_m(a,b,m,x)=(-1)^{m-x}m!(a-m+1)_{m-x}(b-m+1)_x.
\end{equation}

For the difference relations, we refer to \cite{Du2}; we will need the following particular case of (3.10) of \cite{Du2}: for $m=c+1$ we have:

\begin{equation}\label{form10}
(a-x)E_m(a,b,m,x+1) + (x+b-m+1) E_m(a,b,m,x) =0,
\end{equation}

as can be checked directly from \eqref{form9}.

%%%%%%%%%%%%%%%%%%%%%%%%%%%%%%%%%%%%%%%%%%%%%%%%%%%%%%%%%%%%%%%%%%%%%%%%%%%%%%%
%%%%%%%%%%%%%%%%%%%%%%%%%%%%%%%%%%%%%%%%%%%%%%%%%%%%%%%%%%%%%%%%%%%%%%%%%%%%%%%
\subsection{The Johnson scheme.}\label{reprad}
%%%%%%%%%%%%%%%%%%%%%%%%%%%%%%%%%%%%%%%%%%%%%%%%%%%%%%%%%%%%%%%%%%%%%%%%%%%%%%%
%%%%%%%%%%%%%%%%%%%%%%%%%%%%%%%%%%%%%%%%%%%%%%%%%%%%%%%%%%%%%%%%%%%%%%%%%%%%%%%

Fix a positive integer $n$ and denote by $S_n$ the group of all permutations of the set $\{1,2,\dotsc,n\}$ (the symmetric group). For $1\leq m\leq n$, denote by $\Omega_m$ the space of all $m$-subsets of $\{1,2,\dotsc,n\}$; $S_n$ acts on $\Omega_m$ and $\Omega_m\cong S_n/(S_{n-m}\times S_m)$, as homogeneous space. Following the standard notation in the representation theory of $S_n$, denote by $M^{n-m,m}$ the permutation module of all complex valued functions defined on $\Omega_m$.
The space $M^{n-m,m}$ is endowed with the natural scalar product
$\langle f_1,f_2\rangle = \sum_{\omega\in \Omega_a}f_1(\omega)\overline{f_2(\omega)}$,
for $f_1,f_2\in M^{n-m,m}$. For $A\in\Omega_m$, $\delta_A$ will indicate the Dirac function centered at $A$, that is $\delta_A(B)=1$ if $A=B$, $\delta_A(B)=0$ if $A\neq B$.
Now we introduce the Radon transforms $d,d^*$: if $A\in\Omega_m$, we set

\[
d\delta_A=\sum_{x\in A}\delta_
{A\setminus\{x\}}\qquad\quad\text{and}\text\qquad d^*\delta_A=\sum_{x\notin A}\delta_
{A\cup\{x\}}.
\]

We also introduce a particular notation for the powers of $d^*$:

\[
R_q=\frac{(d^*)^q}{q!},\qquad\qquad\text{that is}\qquad 
R_q\delta_A=
\sum_{\substack{B\in \Omega_{m+q}:\\ A\subset B}}\delta_B.
\]

The operator $d$ intertwines the permutation modules $M^{n-m,m}$ and 
$M^{n-m+1,m-1}$ and $d^*$ is the adjoint of $d$. Clearly, $d,d^*,R_q$ may be defined on each $M^{n-m,m}$, for any the value of $n$ and $m$; to simplify notation, we will not indicate the space on which they are acting. 
In the following lemma, we collect some basic properties of the operators $d,d^*,R_q$ \cite{CST1,CST3,Du1,Du2,ScarabottiSabc}. 

\begin{lemma}\label{opd}
If $f\in M^{n-m,m}$ then
\begin{enumerate}
\item\label{opd2} if $1 \leq q\leq n-m$  then  
$dR_qf = R_qdf + 
(n-2m - q + 1)R_{q - 1}f$;
\item\label{opd3} if $1 \leq p \leq q \leq n-m$  and  
$df = 0$  then  
$\left(d\right)^pR_qf = (n-2m - q + 1)_pR_{q-p}f$.

\end{enumerate}
\end{lemma}

In the following Theorem, we give the decomposition of $M^{n-h,h}$ into irreducible $S_n$-representations \cite{CST1,CST3,Du1,Du2}.
 
\begin{theorem}\label{corradtr}
For $0\leq k\leq n/2$, set $S^{n-k,k}=M^{n-k,k}\cap\mbox{\rm Ker}d$.
\begin{enumerate}

\item\label{corradtr0}

$S^{n-k,k}$ is an irreducible $S_n$ representation and its dimension is equal to $\binom{n}{k}-\binom{n}{k-1}$.

\item\label{corradtr1} If $0\leq m\leq n$,  $0\leq k\leq \min\{n-m,m\}$ and $f_1,f_2\in S^{n-k,k}$, then 

\[
\langle R_{m-k}f_1,R_{m-k}f_2\rangle_{M^{n-m,m}}=\binom{n-2k}{m-k}\langle f_1,f_2\rangle_{M^{n-k,k}}
\]

and therefore $R_{m-k}$ is injective from  
$S^{n-k,k}$ to $M^{n-m,m}$;

\item\label{corradtr2} 
\[
M^{n-m,m}=\bigoplus _{k=0}^{\min\{n-m,m\}} R_{m-k}
S^{n-k,k}
\]
is the decomposition of  $M^{n-m,m}$  into 
irreducible $S_n$-representations;

\end{enumerate}
\end{theorem}

%%%%%%%%%%%%%%%%%%%%%%%%%%%%%%%%%%%%%%%%%%%%%%%%%%%%%%%%%%%%
\subsection{An explicit Littlewood-Richardson rule for $S^{n-k,k}$}\label{secLR}
%%%%%%%%%%%%%%%%%%%%%%%%%%%%%%%%%%%%%%%%%%%%%%%%%%%%%%%%%%%

In this section, we study the restriction $\text{Res}^{S_n}_{S_{n-h}\times S_h}S^{n-k,k}$. We introduce the following notation: if $X\subseteq\{1,2,\dotsc,n\}$, $\lvert X\rvert =t$ and $0\leq j\leq t$, then $M^{t-j,j}(X)$ and $S^{t-j,j}(X)$ denote the spaces described in the preceding section constructed by mean of the $j$-subsets of $X$. Set $A=\{1,2,\dotsc,h\}$, $B=A^C\equiv\{h+1,h+2\dotsc,n\}$ and suppose that $S_{n-h}\times S_h$ is the stabilizer of $A$. 

First note that the map $\Omega_s(A)\times\Omega_t(B)\ni(X,Y)\mapsto X\cup Y$ identifies $\Omega_s(A)\times\Omega_t(B)$ with $\{Z\in\Omega_{s+t}:\lvert Z\cap A\rvert=s,\lvert Z\cap B\rvert=t\}$. Moreover, $\Omega_k=\coprod\limits_{l=\max\{0,k-h\}}^{\min\{k,n-h\}}\left[\Omega_{k-l}(A)\times\Omega_{l}(B)\right]$ is the decomposition of $\Omega_k$ into $S_{n-h}\times S_h$-orbits (in the present paper, $\coprod$ denotes a disjoint union). Form this we immediately get a preliminary decomposition:

\begin{equation}\label{SnSnh}
\text{Res}^{S_n}_{S_{n-h}\times S_h}M^{n-k,k}=\bigoplus_{l=\max\{0,k-h\}}^{\min\{k,n-h\}}M^{h-k+l,k-l}(A)\otimes M^{n-h-l,l}(B).
\end{equation}

Note that, in our notation, $\delta_X\otimes\delta_Y=\delta_{X\cup Y}$. Moreover, it
is easy to see that if $\phi\otimes\psi\in M^{n-k+l,k-l}(A)\otimes M^{n-h-l,l}(B)$, then

\begin{equation}\label{dprod}
d[\phi\otimes \psi]=(d\phi)\otimes \psi+\phi\otimes(d\psi);
\end{equation}

in the notation of \cite{Du2}, this is just the usual differentiation rule.
Moreover, \ref{corradtr2} in Theorem \ref{corradtr} ensures us that 
$M^{h-k+l,k-l}(A)\otimes M^{n-h-l,l}(B)$ contains a subspace isomorphic to  $S^{h-i,i}\otimes S^{n-h-j,j}$ if and only if $\max\{i-h+k,j\}\leq l\leq \min\{k-i,n-h-j\}$. In the following Lemma and the subsequent Corollary, we give a characterization of the subspace of $\text{Res}^{S_n}_{S_{n-h}\times S_h}S^{n-k,k}$ isomorphic to $S^{h-i,i}\otimes S^{n-h-j,j}$.

\begin{lemma} Suppose that $0\leq i\leq h/2$, $0\leq j\leq (n-h)/2$ and that
$\phi\in S^{h-i,i}(A)$, $\psi\in S^{n-h-j,j}(B)$. Then the equation (in the unknown coefficients $\alpha_l$'s)

\begin{equation}\label{eqricor}
d\sum_{l=\max\{i-h+k,j\}}^{\min\{k-i,n-h-j\}}\alpha_l(R_{k-l-i}\phi)\otimes(R_{l-j}\psi)=0
\end{equation}

has a nontrivial solution if and only if $k-h\leq j-i\leq n-h-k$ and $i+j\leq k$. If these conditions are satisfied, then all the solutions are obtained by setting $\alpha_l=\alpha E_{k-i-j}(n-h-2j,h-2i,k-i-j,l-j)$, $l=j,i+1,\dotsc,k-i$, $\alpha$ arbitrary constant.

\end{lemma}
\begin{proof}
It is just an application of \eqref{dprod} and Lemma \ref{opd}.
We can divide the study of \eqref{eqricor} in four cases, according with the values of the limits in the sum.\\

In the first case we take $i-h+k\leq j$ and $k-i\leq n-h-j$, that is $k-h\leq j-i\leq n-h-k$; clearly, we must also have $i+j\leq k$. Now

\begin{multline}\label{dalphaphipsi1}
d\sum_{l=j}^{k-i}\alpha_l(R_{k-l-i}\phi)\otimes(R_{l-j}\psi)\\
=\sum_{l=j}^{k-i-1}\alpha_l(h-k-i+l+1)(R_{k-l-i-1}\phi)\otimes(R_{l-j}\psi)+\sum_{l=j+1}^{k-i}\alpha_l(n-h-j-l+1)(R_{k-l-i}\phi)\otimes(R_{l-j-1}\psi)\\
=\sum_{l=j}^{k-i-1}[\alpha_l(h-k-i+l+1)+\alpha_{l+1}(n-h-j-l)](R_{k-l-i-1}\phi)\otimes(R_{l-j}\psi).\\
\end{multline}

Therefore \eqref{eqricor} is satisfied if and only if

\[
\alpha_l(h-k-i+l+1)+\alpha_{l+1}(n-h-j-l)=0,\qquad l=j,j+1,\dotsc,k-i-1,
\]

and this is solved by setting $\alpha_l=E_{k-i-j}(n-h-2j,h-2i,k-i-j,l-j)$ (see \eqref{form10}). \\

In the second case, we take $j <i-h+k$ and $k-i\leq n-h-j$. Now we have

\begin{multline}\label{dalphaphipsi2}
d\sum_{l=i-h+k}^{k-i}\alpha_l
(R_{k-l-i}\phi\otimes R_{l-j}\psi)\\
=\sum_{l=i-h+k}^{k-i-1}\alpha_l(h-k-i+l+1)(R_{k-l-i-1}\phi)\otimes(R_{l-j}\psi)+
\sum_{i-h+k}^{k-i}\alpha_l(n-h-j-l+1)(R_{k-l-i}\phi)\otimes(R_{l-j-1}\psi)\\
=\sum_{i-h+k}^{k-i-1}[\alpha_l(h-k-i+l+1)+\alpha_{l+1}(n-h-j-l)](R_{k-l-i-1}\phi)\otimes(R_{l-j}\psi)\\
+\alpha_{i-h+k}(n+k-i-j+1)(R_h\phi)\otimes(R_{i-h+k-j-1}\psi)
\end{multline}

and this is again equivalent to $\alpha_l(h-k-i+l+1)+\alpha_{l+1}(n-h-j-l)=0$, but now we have the extra condition $\alpha_{i-h+k}=0$, and therefore we have only the trivial solution.

In the same way,
it easy to show that also if we take $j\geq i-h+k$ and $n-h-j<k-i$, or $j< i-h+k$ and $n-h-j<k-i$, we have only the trivial solution. 

\end{proof}

An immediate consequence is an explicit form of Littlewood-Richardson rule for $S^{n-k,k}$  (see \cite{Ja2} for the general rule).

\begin{corollary}\label{LR}
The multiplicity of $S^{h-i,i}\otimes S^{n-h-j,j}$ in  $\text{Res}^{S_n}_{S_{n-h}\times S_h}S^{n-k,k}$ is equal to 1 if $k-h\leq j-i\leq n-h-k$ and $i+j\leq k$; otherwise it is equal to 0. 

In particular,

\[
\text{Res}^{S_n}_{S_{n-h}\times S_h}S^{n-k,k}=\bigoplus S^{h-i,i}\otimes S^{n-h-j,j}
\]

where the sum is over all $i,j$ satisfying the above conditions (and the decomposition is multiplicity free).
Moreover, the map 

\[
\phi\otimes\psi\quad\longmapsto\quad\sum_{l=j}^{k-i}E_{k-i-j}(n-h-2j,h-2i,k-i-j,l-j) (R_{k-l-i}\phi)\otimes (R_{l-j}\psi)
\]

is an explicit immersion of $S^{h-i,i}\otimes S^{n-h-j,j}$ into $\text{Res}^{S_n}_{S_{n-h}\times S_h}S^{n-k,k}$.

\end{corollary}

\begin{proof}
Just note that by \eqref{SnSnh}, only representations of the form $S^{h-i,i}\otimes S^{n-h-j,j}$ can appear in the decomposition of $\text{Res}^{S_n}_{S_{n-h}\times S_h}S^{n-k,k}$ into irreducible representations.
\end{proof}

\begin{lemma}\label{Rmk}
Suppose again that $0\leq i\leq h/2$, $0\leq j\leq (n-h)/2$ and that
$\phi\in S^{h-i,i}(A)$, $\psi\in S^{n-h-j,j}(B)$. For $0\leq k\leq m\leq n $ and $\max\{i-h+k,j\}\leq l\leq \min\{k-i,n-h-j\}$, we have

\[
R_{m-k}\left[(R_{k-l-i}\phi)\otimes (R_{l-j}\psi)\right]=\sum_{w=\max\{l,i-h+m\}}^{\min\{m-k+l,n-h-j\}}\binom{m-w-i}{k-l-i}\binom{w-j}{l-j}(R_{m-w-i}\phi)\otimes (R_{w-j}\psi)
\]

\end{lemma}

\begin{proof}
First of all, note that $d^*(\phi'\otimes\psi')=(d^*\phi')\otimes\psi'+\phi'\otimes(d^*\psi')$ (analogous to \eqref{dprod}).
Therefore, as in the Leibnitz rule of elementary calculus, we get

\[
(d^*)^t(\phi'\otimes\psi')=\sum_{s=0}^t\binom{t}{s}\left[(d^*)^{t-s}\phi'\right]\otimes[(d^*)^s\psi']. 
\]

Since $R_qR_p=\binom{p+q}{q}R_{p+q}$, it follows that

\begin{equation}\label{RtRqRp}
R_t\left[(R_q\phi')\otimes(R_p\psi')\right]=\sum_{s=0}^t\binom{t-s+q}{q}\binom{s+p}{p}\left[(R_{t-s+q}\phi')\otimes(R_{s+p}\psi')\right]. 
\end{equation}

 The formula in the statement is obtained by setting $t=m-k,q=k-l-i,p=l-j$ and $s=w-l$ (and taking into account \ref{corradtr2} in Theorem \ref{corradtr} for the limits in the sum).
 
\end{proof}

\begin{theorem}\label{LRexplicit} 
For $0\leq m,h\leq n$, $0\leq k\leq\min\{n-m,m\}$, $0\leq i\leq h/2$, $0\leq j\leq (n-h)/2$,
$k-h\leq j-i\leq n-h-k$ and $i+j\leq k$, the map 
$T_m:S^{h-i,i}(A)\otimes S^{n-h-j,j}(B)\mapsto M^{n-m,m}$ given by setting, for $\phi\otimes\psi\in S^{h-i,i}(A)\otimes S^{n-h-j,j}(B)$,

\begin{multline*}
T_m(\phi\otimes\psi)=\sum_{w=\max\{j,i-h+m\}}^{\min\{n-h-j,m-i\}}E_{k-i-j}(n-h-2j,h-2i,m-i-j,w-j)(R_{m-w-i}\phi)\otimes (R_{w-j}\psi),
\end{multline*}

is an explicit immersion of $S^{h-i,i}\otimes S^{n-h-j,j}$ into $\text{\rm Res}^{S_n}_{S_{n-h}\times S_h}\left[R_{m-k}S^{n-k,k}\right]$.
Moreover, 

\begin{equation}\label{norm}
\begin{split}
\lVert T_m(\phi\otimes\psi)\rVert^2_{M^{n-m,m}}&=\binom{n-2i-2j}{m-i-j}\binom{n-2i-2j}{k-i-j}^{-1}\cdot\frac{n-i-j-k+1}{n-2k+1}\times\\
&\times(n-h-j-k+i+1)_{k-i-j}(h-i-k+j+1)_{k-i-j}(m-k+1)_{k-i-j}\times\\ &\times(n-m-k+1)_{k-i-j}\lVert\phi\rVert^2_{M^{h-i,i}(A)}\lVert\psi\rVert^2_{M^{n-h-j,j}(B)}.
\end{split}
\end{equation}

Finally,

\begin{equation}\label{RqTm}
R_qT_m=\binom{m-k+q}{q}T_{m+q}.
\end{equation}

\end{theorem}
\begin{proof}

First note that 

\begin{equation}\label{ljki}
\sum_{l=j}^{k-i}\;\sum_{w=\max\{l,i-h+m\}}^{\min\{m-k+l,n-h-j\}}=\sum_{w=\max\{j,i-h+m\}}^{\min\{n-h-j,m-i\}}\sum_{l=\max\{j,w-m+k\}}^{\min\{k-i,w\}}.
\end{equation}
 
Then by composing the map in Corollary \ref{LR} with $R_{m-k}$ (see \ref{corradtr2} in Theorem \ref{corradtr}), we get

\begin{multline*}
R_{m-k}\sum_{l=j}^{k-i}E_{k-i-j}(n-h-2j,h-2i,k-i-j,l-j)(R_{k-i-l}\phi)\otimes(R_{l-j}\psi)\\
=\sum_{w=\max\{j,i-h+m\}}^{\min\{n-h-j,m-i\}}\Biggl[\sum_{l=\max\{j,w-m+k\}}^{\min\{k-i,w\}}\binom{m-w-i}{k-l-i}\binom{w-j}{l-j}\times\\
\times E_{k-i-j}(n-h-2j,h-2i,k-i-j,l-j)\Biggr](R_{m-w-i}\phi)\otimes (R_{w-j}\psi)\\
=\sum_{w=\max\{j,i-h+m\}}^{\min\{n-h-j,m-i\}}E_{k-i-j}(n-h-2j,h-2i,m-i-j,w-j)(R_{m-w-i}\phi)\otimes (R_{w-j}\psi)
\end{multline*}

where the first equality follows from Lemma \ref{Rmk} and \eqref{ljki} and the second equality from \eqref{form7}. This proves that $T_m$ is an immersion of $S^{h-i,i}\otimes S^{n-h-j,j}$ into $\text{Res}^{S_n}_{S_{n-h}\times S_h}\left[R_{m-k}S^{n-k,k}\right]$.

From \ref{corradtr1} in Theorem \ref{corradtr}, we deduce that 

\[
\lVert(R_{m-w-i}\phi)\otimes(R_{w-j}\psi)\rVert^2_{M^{n-m,m}}=\binom{h-2i}{m-w-i}\binom{n-h-2j}{w-j}\lVert\phi\rVert^2_{S^{h-i,i}(A)}\lVert\psi\rVert^2_{S^{n-h-j,j}(B)}
\]

and then \eqref{norm} follows from the orthogonality relations \eqref{orthrel}.

Finally, \eqref{RqTm} is a simple consequence of the identities $T_m=R_{m-k}T_k$ and $R_qR_{m-k}=\binom{m-k+q}{q}R_{m-k+q}$.

\end{proof}

If $X\subseteq\{1,2,\dotsc,n\}$ and $0\leq t\leq\lvert X\rvert$, we denote by  $\sigma_t(X)$ the characteristic function of the set of all $Y\subseteq X,\lvert Y\rvert=t$. 
If $\max\{0,m-h\}\leq w\leq\min\{n-h,m\}$, we denote by $\sigma_{m-w}(A)\otimes\sigma_w(B)$ the characteristic function of the set of all $Y\in\Omega_m$ such that $\lvert Y\cap A\rvert=m-w$ and $\lvert Y\cap B\rvert=w$ (compare with \eqref{SnSnh}).

\begin{corollary}\label{spherical}\cite{Du2} For $0\leq m,h\leq n$ and $0\leq k\leq \min\{n-m,m,n-h,h\}$, the space of $S_{n-h}\times S_h$-invariant vectors in the representation  $R_{m-k}S^{n-k,k}$ is spanned by the function $\Phi(n,h,m,k)=\sum_{w=\max\{0,-h+m\}}^{\min\{n-h,m\}}E_k(n-h,h,m,w)\sigma_{m-w}(A)\otimes\sigma_w(B)$.
\end{corollary}

\begin{proof}
Take $\phi=\sigma_0(A)$ and $\psi=\sigma_0(B)$ (and therefore $i=j=0$) in Theorem \ref{LRexplicit}.
\end{proof}

In particular, the spherical functions of the Gelfand pair $(S_n,S_{n-m}\times S_m)$, normalized so that the coefficient of $\sigma_{m}(A)\sigma_0(B)$ is 1, are given by:
$\frac{1}{(-1)^k(n-m-k+1)_k(m-k+1)_k}\Phi(n,m,m,k)$, $0\leq k\leq\min\{n-m,m\}$.

%%%%%%%%%%%%%%%%%%%%%%%%%%%%%%%%%%%%%%%%%%%%%%%%%%%%%%%%%%%%%%%%%%%%%%%
%%%%%%%%%%%%%%%%%%%%%%%%%%%%%%%%%%%%%%%%%%%%%%%%%%%%%%%%%%%%%%%%%%%%
\section{The tree method for multidimensional Hahn polynomials}
%%%%%%%%%%%%%%%%%%%%%%%%%%%%%%%%%%%%%%%%%%%%%%%%%%%%%%%%%%%%%%%%%%%
%%%%%%%%%%%%%%%%%%%%%%%%%%%%%%%%%%%%%%%%%%%%%%%%%%%%%%%%%%%%%%%%%%%%

%%%%%%%%%%%%%%%%%%%%%%%%%%%%%%%%%%%%%%%%%%%%%%%%%%%%%%%%%%%%%%%%%%%%%%%
%%%%%%%%%%%%%%%%%%%%%%%%%%%%%%%%%%%%%%%%%%%%%%%%%%%%%%%%%%%%%%%%%%%%
\subsection{Labeled trees and multidimensional Hahn polynomials}\label{Labeled}
%%%%%%%%%%%%%%%%%%%%%%%%%%%%%%%%%%%%%%%%%%%%%%%%%%%%%%%%%%%%%%%%%%%
%%%%%%%%%%%%%%%%%%%%%%%%%%%%%%%%%%%%%%%%%%%%%%%%%%%%%%%%%%%%%%%%%%%%

We recall that a {\em tree} is a connected simple graph without loops or circuits.
A {\em rooted binary tree} is a tree with a distinguished vertex $\alpha$ ({\em the root}) of degree 2 and all the remaining vertices of degree 3 or 1. The vertices of degree 1 are called {\em the leaves}, all the other vertices are called {\em internal vertices}, or {\em branch points}. In what follows, $\mathcal{T}$ is always a finite binary tree (identified with the set of its vertices).

The $l$-th {\em level} of a tree $\mathcal{T}$, denoted by $\mathcal{T}_l$, is formed by the vertices at distance $l$ from the root. The {\em height} of $\mathcal{T}$ is the greatest $L$ such that there exists a vertex in $\mathcal{T}$ at distance $L$ from the root. If $\alpha\in\mathcal{T}_l$ is an internal vertex, then there exist exactly two vertices $\beta,\gamma\in\mathcal{T}_{l+1}$ connected with $\alpha$; they are called the {\em sons} of $\alpha$, while $\alpha$ is the father of $\beta$ and $\gamma$. We think of $\mathcal{T}$ as a {\em planar} tree, and therefore $\alpha$ has a {\em left} son and a {\em right} son.  In the figure below, $\beta$ is the left son and $\gamma$ is the right son.

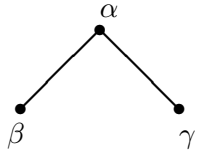
\begin{picture}(400,80)
\put(200,50){\circle*{4}}
\put(170,20){\circle*{4}}
\put(230,20){\circle*{4}}
\put(200,55){$\alpha$}
\put(165,7){$\beta$}
\put(230,7){$\gamma$}

\thicklines
\put(200,50){\line(-1,-1){30}}
\put(200,50){\line(1,-1){30}}

\end{picture}

For a tree $\mathcal{T}$, we denote by $\mathcal{T}'$ and $\mathcal{T}''$ the subtrees formed respectively by the left descendants and the right descendants of the root. 
We denote by $\alpha$ the root of $\mathcal{T}$, and by $\beta$ and $\gamma$ respectively its left and right son. Then $\beta$ is the root of $\mathcal{T}'$ and $\gamma$ is the root of $\mathcal{T}''$. This decomposition will be the key for the following iterative procedure.\\ 

{\bf Basic iterative procedures.}
\begin{itemize}
\item
We define/prove something for the subtrees $\mathcal{T}'$ and $\mathcal{T}''$, examining the particular cases in which $\beta$ or/and $\gamma$ is a leaf, and then we show how to pass to the entire $\mathcal{T}$.\\
\item
We define/prove something for the root $\alpha$ and then we show how to pass to its sons $\beta$ and $\gamma$. 
\end{itemize}

Now we show how to use the second procedure to label the tree.
We will use four type of labelings. \\

The first is the {\em composition labeling}.
Suppose that ${\bf a}=(a_1,a_2,\dotsc,a_h)$ is a composition of $n$. This means that $a_1,a_2,\dotsc,a_h$ are positive integers and that $a_1+a_2+\dotsb +a_h=n$; $a_1,a_2,\dotsc,a_h$ are the {\em parts} of ${\bf a}$. We will write ${\bf a}\Vdash n$ to denote that ${\bf a}$ is a composition of $n$. Suppose also that $\mathcal{T}$ has $h$ leaves. We denote by $\mathcal{T}({\bf a})$ the labeled tree obtained in the following recursive way. The label of $\alpha$ is $(a_1,a_2,\dotsc,a_h)$. If $\mathcal{T}'$ has $t$ leaves, then  the label of its root $\beta$ is ${\bf a}'=(a_1,a_2,\dotsc,a_t)$, while the label of $\gamma$ (the root of $\mathcal{T}''$) is ${\bf a}''=(a_{t+1},a_{t+2},\dotsc,a_h)$. 
Note that given ${\bf a}$ and $\mathcal{T}$, the labeling is unique.
We give two examples with $h=4$.

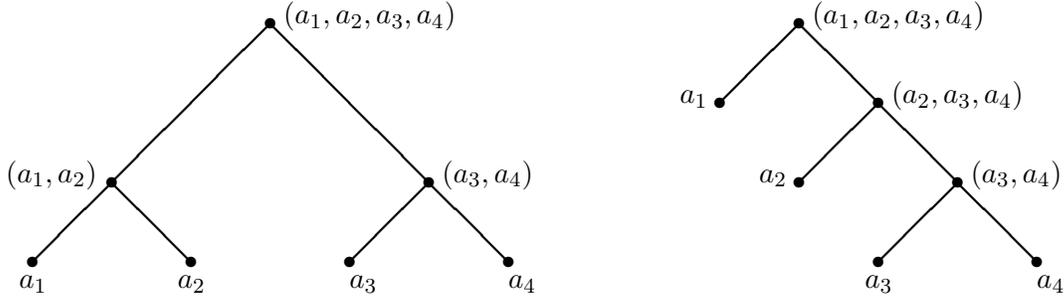
\begin{picture}(400,150)
%%%%%%%%%%%%%%% zero level
%%%%%%%%%%%%%%% first tree
\put(100,120){\circle*{4}}
%%%%%% labelings
\put(105,120){$(a_1,a_2,a_3,a_4)$}

%%%%%%%%%%%%%%second tree
\put(300,120){\circle*{4}}
%%%%%% labelings
\put(305,120){$(a_1,a_2,a_3,a_4)$}

%%%%%%%%%%%% first level 

%%%%%%%%%%%% second tree
\put(270,90){\circle*{4}}
\put(330,90){\circle*{4}}

\thicklines
\put(300,120){\line(-1,-1){30}}
\put(300,120){\line(1,-1){90}}

%%%%%%%%%% labelings
\put(255,90){$a_1$}
\put(335,90){$(a_2,a_3,a_4)$}

%%%%%%%%%%%%%%%  second level

%%%%%%%%%%%%%%%  first tree
\put(160,60){\circle*{4}}
\put(40,60){\circle*{4}}

\put(100,120){\line(-1,-1){90}}
\put(100,120){\line(1,-1){90}}

%%%%%%%%%%%% labelings
\put(165,60){$(a_3,a_4)$}
\put(0,60){$(a_1,a_2)$}

%%%%%%%%%%%%% second tree
\put(300,60){\circle*{4}}
\put(360,60){\circle*{4}}

\put(330,90){\line(-1,-1){30}}

%%%%%%%%%%%% labelings
\put(365,60){$(a_3,a_4)$}
\put(285,60){$a_2$}
%%%%%%%%%%%%%%%%%%% third level

%%%%%%%%%%%%%%%%%%% first tree
\put(130,30){\circle*{4}}
\put(190,30){\circle*{4}}
\put(70,30){\circle*{4}}
\put(10,30){\circle*{4}}

\put(40,60){\line(1,-1){30}}
\put(160,60){\line(-1,-1){30}}

%%%%%%%%%%%%%%%%%% labelings
\put(130,20){$a_3$}
\put(190,20){$a_4$}
\put(65,20){$a_2$}
\put(5,20){$a_1$}

%%%%%%%%%%%%%%%%%% second tree

\put(330,30){\circle*{4}}
\put(390,30){\circle*{4}}

\put(360,60){\line(-1,-1){30}}
%%%%%%%%%%%%%%%% labelings
\put(325,20){$a_3$}
\put(390,20){$a_4$}

\end{picture}

In what follows, to simplify notation, we set $\underline{a}_t=a_1+a_2+\dotsb+a_t$ (and therefore ${\bf a}'\Vdash \underline{a}_t$ and ${\bf a}''\Vdash n-\underline{a}_t$).

The second is the {\em variables labeling}. It is defined simply by taking a set of variables indicized by the internal vertices of $\mathcal{T}$. The resulting labeled tree is denote by $\mathcal{T}_{\text{v}}({\bf w})$, where ${\bf w}=(w_\tau)_{\tau\in\mathcal{T}}$ is a vector of variables (and $w_\tau$ is the label of $\tau\in\mathcal{T}$; often, the index $\tau$ will be omitted). The leaves do not have labels.\\

The third is the {\em spaces labeling}. It is defined by choosing a label $0\leq m\leq n$ for the root $\alpha$; if $w$ is the variable associated to the root, then $m-w$ is the label associated to $\beta$ and $w$ is the label associated to $\gamma$; the leaves are not labeled. The resulting labeled tree is denoted by $\mathcal{T}_{\text{s}}(m)$.\\

The fourth is the {\em representations labeling}. It is based on Corollary \ref{LR}. The label of a leaf is zero.   Suppose that 
$k$ is the label of $\alpha$, $i$ is the label of $\beta$ and $j$ is the label of $\gamma$. 

\begin{itemize}

\item If $\beta$ and $\gamma$ are both internal we must have

\begin{equation}\label{knm}
0\leq k\leq\min\{n-m,m\}
\end{equation} 

and

\begin{equation}\label{ijk}
\begin{split}
&\quad\qquad 0\leq i\leq\underline{a}_t/2,\quad\qquad 0\leq j\leq(n-\underline{a}_t)/2,\\
&k-\underline{a}_t\leq j-i\leq n-\underline{a}_t-k,\qquad i+j\leq k.
\end{split}
\end{equation}

\item Suppose $\beta$ is a leaf (resp. $\gamma$ is a leaf) and $\gamma$ is internal (resp. $\beta$ is internal). Now $t=1$, $i= 0$ and $k,j$ must satisfy the conditions in \eqref{knm} and \eqref{ijk}, with $i=0$ (resp. $t=h-1$, $j=0$ and $k,i$  must satisfy those conditions with $j=0$).\\

\item Suppose that $\beta$ and $\gamma$ are both leaves: now $h=2$ and \eqref{knm},\eqref{ijk} reduce to $0\leq k\leq\min\{a_1,a_2\}$.

We may also say that \eqref{knm},\eqref{ijk} must be satisfied and that we must have $i=0$ (resp. $j=0$) if $\beta$ (resp. $\gamma$) is a leaf. The resulting labeled tree will be denoted by $\mathcal{T}_{\text{r}}({\bf k})$. 

\end{itemize}

Clearly, the spaces labeling depends on $n$ and the variable labeling, while the representation labeling depends on both the composition and the spaces labeling.

Once we have labeled $\mathcal{T}$ as above, then $\mathcal{T}'({\bf a}),\mathcal{T}''({\bf a}),\dotsc$ will denote the subtrees $\mathcal{T}'$ and $\mathcal{T}''$ with the labeling inherited by $\mathcal{T}$.\\

Now we define a set of multidimensional Hahn polynomials associated to the tree $\mathcal{T}$ with the labelings defined above. We keep all the preceding notation; the definition is recursive.

\begin{itemize}

\item For 

\begin{equation}\label{jiat}
\max\{j,i-\underline{a}_t+m\}\leq w\leq\min\{n-\underline{a}_t-j,m-i\},
\end{equation}

we set

\begin{multline*}
E_{\mathcal{T}_{\text{\rm r}}({\bf k})}(\mathcal{T}({\bf a}),m,\mathcal{T}_{\text{\rm v}}({\bf w}))=
E_{k-i-j}(n-\underline{a}_t-2j,\underline{a}_t-2i,m-i-j,w-j)\times\\
\times E_{\mathcal{T}'_{\text{\rm r}}({\bf k})}(\mathcal{T}'({\bf a}),m-w,\mathcal{T}'_{\text{\rm v}}({\bf w}))
E_{\mathcal{T}''_{\text{\rm r}}({\bf k})}(\mathcal{T}''({\bf a}),w,\mathcal{T}''_{\text{\rm v}}({\bf w}))
\end{multline*}

\item If $\beta$ is a leaf, then $w$ must satisfy \eqref{jiat} with $i=0$, and we set $E_{\mathcal{T}'_{\text{\rm r}}({\bf k})}(\mathcal{T}'({\bf a}),m-w,\mathcal{T}'_{\text{\rm v}}({\bf w}))=1$.

\item If $\gamma$ is a leaf, then  
$w$ must satisfy \eqref{jiat} with $j=0$,
and we set $E_{\mathcal{T}''_{\text{\rm r}}({\bf k})}(\mathcal{T}''({\bf a}),w,\mathcal{T}''_{\text{\rm v}}({\bf w}))=1$.

\item If $\beta$ and $\gamma$ are both leaves,  
$w$ must satisfy \eqref{jiat} with $i,j=0$, 
and

\[
E_{\mathcal{T}_{\text{\rm r}}({\bf k})}(\mathcal{T}({\bf a}),m,\mathcal{T}_{\text{\rm v}}({\bf w}))=
E_k(a_2,a_1,m,w).
\]

\end{itemize}

\begin{remark}{\rm
The conditions \eqref{knm},\eqref{ijk} and \eqref{jiat} are imposed by Corollary \ref{LR} and Theorem \ref{LRexplicit} and agree with \eqref{cab}. Moreover, when \eqref{jiat} is satisfied, we have automatically $0\leq i\leq\min\{\underline{a}_t-m+w,m-w\}$ and $0\leq j\leq\min\{n-\underline{a}_t-w,w\}$, that is we do not need to 
impose the analog of \eqref{knm} to $i,j$. In other words, once we have imposed \eqref{jiat}, the labels in $\mathcal{T}_\text{r}({\bf k})$ must satisfy only the conditions obtained by iterating \eqref{ijk}, while the variables in $\mathcal{T}_\text{v}({\bf w})$ must satisfy only the conditions obtained by iterating \eqref{jiat}. 
}
\end{remark}

\begin{example}\label{h4}{\rm

Suppose that $h=4$ and take the tree and the labelings below. 

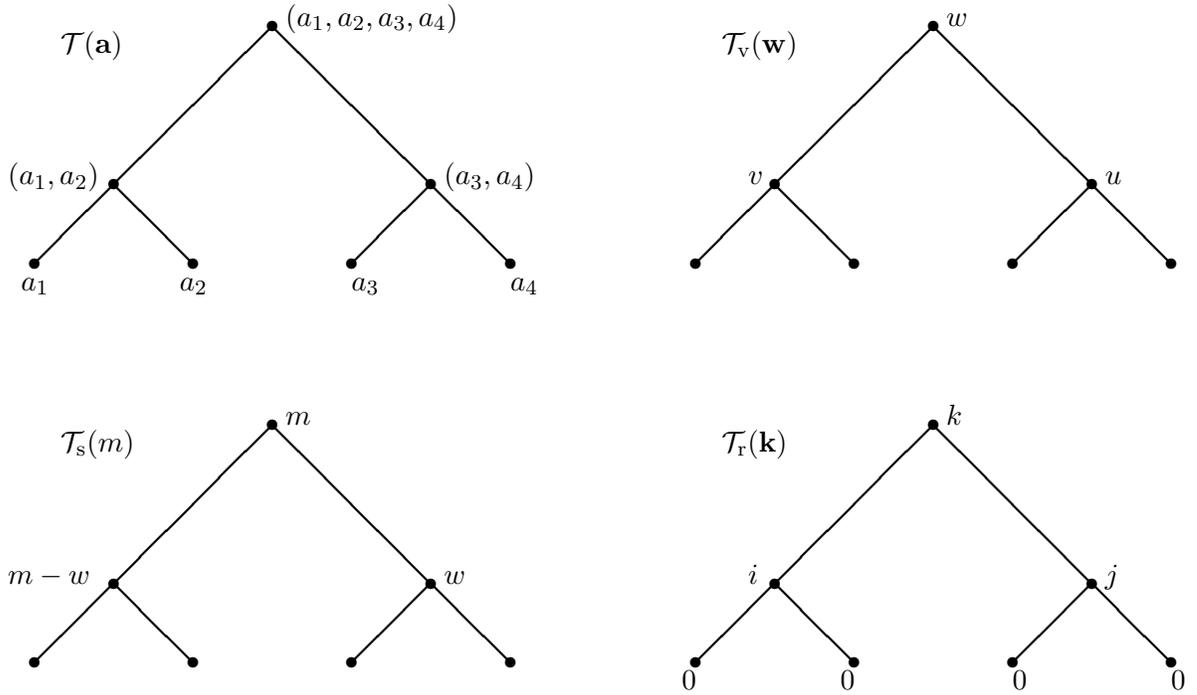
\begin{picture}(450,150)
\put(20,110){$\mathcal{T}({\bf a})$}
\put(270,110){$\mathcal{T}_{\text{v}}({\bf w})$}
\thicklines
%%%%%%%%%%%%%%% zero level
%%%%%%%%%%%%%%% first tree
\put(100,120){\circle*{4}}
%%%%%% labelings
\put(105,120){$(a_1,a_2,a_3,a_4)$}

%%%%%%%%%%%%%%second tree
\put(350,120){\circle*{4}}
%%%%%% labelings
\put(355,120){$w$}

%%%%%%%%%%%% first level 

%%%%%%%%%%%%%%%  first tree
\put(160,60){\circle*{4}}
\put(40,60){\circle*{4}}

\put(100,120){\line(-1,-1){90}}
\put(100,120){\line(1,-1){90}}

%%%%%%%%%%%% labelings
\put(165,60){$(a_3,a_4)$}
\put(0,60){$(a_1,a_2)$}

%%%%%%%%%%%%% second tree
\put(410,60){\circle*{4}}
\put(290,60){\circle*{4}}

\put(350,120){\line(-1,-1){90}}
\put(350,120){\line(1,-1){90}}

%%%%%%%%%%%% labelings
\put(415,60){$u$}
\put(280,60){$v$}

%%%%%%%%%%%%%%%%%%% third level

%%%%%%%%%%%%%%%%%%% first tree
\put(130,30){\circle*{4}}
\put(190,30){\circle*{4}}
\put(70,30){\circle*{4}}
\put(10,30){\circle*{4}}

\put(40,60){\line(1,-1){30}}
\put(160,60){\line(-1,-1){30}}

%%%%%%%%%%%%%%%%%% labelings
\put(130,20){$a_3$}
\put(190,20){$a_4$}
\put(65,20){$a_2$}
\put(5,20){$a_1$}

%%%%%%%%%%%%%%%%%% second tree

\put(380,30){\circle*{4}}
\put(440,30){\circle*{4}}
\put(320,30){\circle*{4}}
\put(260,30){\circle*{4}}

\put(290,60){\line(1,-1){30}}
\put(410,60){\line(-1,-1){30}}

\end{picture}

\begin{picture}(450,150)
\put(20,110){$\mathcal{T}_{\text{s}}(m)$}
\put(270,110){$\mathcal{T}_{\text{r}}({\bf k})$}
\thicklines
%%%%%%%%%%%%%%% zero level
%%%%%%%%%%%%%%% first tree
\put(100,120){\circle*{4}}
%%%%%% labelings
\put(105,120){$m$}

%%%%%%%%%%%%%%second tree
\put(350,120){\circle*{4}}
%%%%%% labelings
\put(355,120){$k$}

%%%%%%%%%%%% first level 

%%%%%%%%%%%%%%%  first tree
\put(160,60){\circle*{4}}
\put(40,60){\circle*{4}}

\put(100,120){\line(-1,-1){90}}
\put(100,120){\line(1,-1){90}}

%%%%%%%%%%%% labelings
\put(165,60){$w$}
\put(0,60){$m-w$}

%%%%%%%%%%%%% second tree
\put(410,60){\circle*{4}}
\put(290,60){\circle*{4}}

\put(350,120){\line(-1,-1){90}}
\put(350,120){\line(1,-1){90}}

%%%%%%%%%%%% labelings
\put(415,60){$j$}
\put(280,60){$i$}

%%%%%%%%%%%%%%%%%%% third level

%%%%%%%%%%%%%%%%%%% first tree
\put(130,30){\circle*{4}}
\put(190,30){\circle*{4}}
\put(70,30){\circle*{4}}
\put(10,30){\circle*{4}}

\put(40,60){\line(1,-1){30}}
\put(160,60){\line(-1,-1){30}}

%%%%%%%%%%%%%%%%%% second tree

\put(380,30){\circle*{4}}
\put(440,30){\circle*{4}}
\put(320,30){\circle*{4}}
\put(260,30){\circle*{4}}

\put(290,60){\line(1,-1){30}}
\put(410,60){\line(-1,-1){30}}

\put(380,20){0}
\put(440,20){0}
\put(315,20){0}
\put(255,20){0}

\end{picture}

Then, for $0\leq m\leq n\equiv a_1+a_2+a_3+a_4$, $0\leq k\leq \min\{n-a_1,n-a_2,n-a_3,n-a_4,m,n-m\}$, 
$0\leq i\leq\min\{a_1,a_2\}$, $0\leq j\leq\min\{a_3,a_4\}$,
$k-a_1-a_2\leq j-i\leq a_3+a_4-k$ and $i+j\leq k$, the associated Hahn polynomials are

\begin{multline}
E_{\mathcal{T}_{\text{\rm r}}({\bf k})}(\mathcal{T}({\bf a}),m,\mathcal{T}_{\text{\rm v}}({\bf w}))=
E_{k-i-j}(a_3+a_4-2j,a_1+a_2-2i,m-i-j,w-j)\times\\
\times E_i(a_2,a_1,m-w,v)
E_j(a_4,a_3,w,u)
\end{multline}

defined for
$\max\{j,i-a_1-a_2+m\}\leq w\leq\min\{a_3+a_4-j,m-i\}$, $\max\{0,m-a_1-w\}\leq v\leq\min\{a_2,m-w\}$ and $\max\{0,w-a_3\}\leq u\leq\min\{a_4,w\}$. The conditions $k\leq n-a_r$, $r=1,2,3,4$, come from the conditions on $i,j$ and agree with the Young rule for the symmetric group.
}
\end{example}

\begin{example}\label{ah}{\rm
Now we take a general composition ${\bf a}=(a_1,a_2,\dotsc,a_h)$ but a special kind of tree. To simplify notation, we set $w_0=m$.

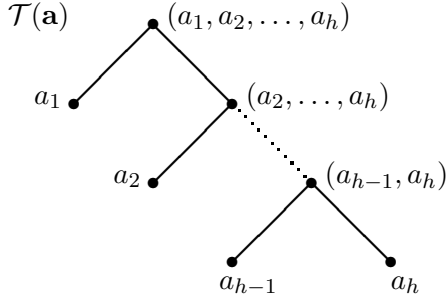
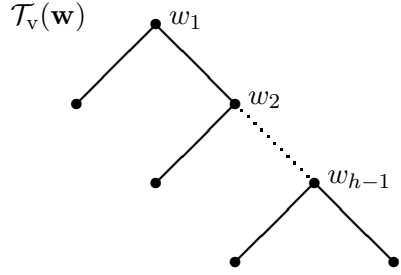
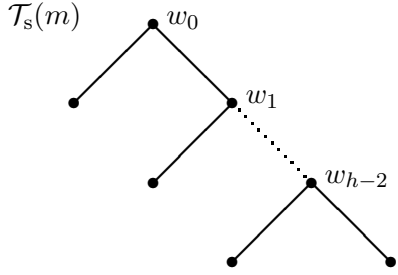
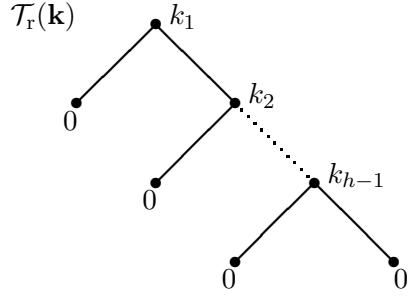
\begin{picture}(400,150)
%%%%%%%%%%%%%%% zero level
\put(-5,120){$\mathcal{T}({\bf a})$}
\put(195,120){$\mathcal{T}_{\text{v}}({\bf w})$}
%%%%%%%%%%%%%%% first tree
\put(50,120){\circle*{4}}
%%%%%% labelings
\put(55,120){$(a_1,a_2,\dotsc,a_h)$}

%%%%%%%%%%%% second tree
\put(250,120){\circle*{4}}
%%%%%% labelings
\put(255,120){$w_1$}

%%%%%%%%%%%% first level 

%%%%%%%%%%%% first tree
\put(20,90){\circle*{4}}
\put(80,90){\circle*{4}}

\thicklines
\put(50,120){\line(-1,-1){30}}
\put(50,120){\line(1,-1){30}}
\put(110,60){\line(1,-1){30}}
\dottedline{4}(80,90)(110,60)

%%%%%%%%%% labelings
\put(5,90){$a_1$}
\put(85,90){$(a_2,\dotsc,a_h)$}

%%%%%%%%%%% second tree
\put(220,90){\circle*{4}}
\put(280,90){\circle*{4}}

\thicklines
\put(250,120){\line(-1,-1){30}}
\put(250,120){\line(1,-1){30}}
\put(310,60){\line(1,-1){30}}
\dottedline{4}(280,90)(310,60)

%%%%%%%%%% labelings

\put(285,90){$w_2$}

%%%%%%%%%%%%%%%  second level

%%%%%%%%%%%%%%%  first tree
\put(50,60){\circle*{4}}
\put(110,60){\circle*{4}}

\put(80,90){\line(-1,-1){30}}

%%%%%%%%%%%% labelings
\put(115,60){$(a_{h-1},a_h)$}
\put(35,60){$a_2$}

%%%%%%%%%%%%%%second tree
\put(250,60){\circle*{4}}
\put(310,60){\circle*{4}}

\put(280,90){\line(-1,-1){30}}

%%%%%%%%%%%% labelings
\put(315,60){$w_{h-1}$}

%%%%%%%%%%%%%%%%%%% third level

%%%%%%%%%%%%%%%%%%% first tree

\put(80,30){\circle*{4}}
\put(140,30){\circle*{4}}

\put(110,60){\line(-1,-1){30}}
%%%%%%%%%%%%%%%% labelings
\put(75,20){$a_{h-1}$}
\put(140,20){$a_h$}

%%%%%%%%%%%%%%%%%%% second tree

\put(280,30){\circle*{4}}
\put(340,30){\circle*{4}}

\put(310,60){\line(-1,-1){30}}

\end{picture}

\begin{picture}(400,150)
%%%%%%%%%%%%%%% zero level
\put(-5,120){$\mathcal{T}_{\text{s}}(m)$}
\put(195,120){$\mathcal{T}_{\text{r}}({\bf k})$}
%%%%%%%%%%%%%%% first tree
\put(50,120){\circle*{4}}
%%%%%% labelings
\put(55,120){$w_0$}

%%%%%%%%%%%% second tree
\put(250,120){\circle*{4}}
%%%%%% labelings
\put(255,120){$k_1$}

%%%%%%%%%%%% first level 

%%%%%%%%%%%% first tree
\put(20,90){\circle*{4}}
\put(80,90){\circle*{4}}

\thicklines
\put(50,120){\line(-1,-1){30}}
\put(50,120){\line(1,-1){30}}
\put(110,60){\line(1,-1){30}}
\dottedline{4}(80,90)(110,60)
%%%%%%%%%% labelings

\put(85,90){$w_1$}

%%%%%%%%%%% second tree
\put(220,90){\circle*{4}}
\put(280,90){\circle*{4}}

\thicklines
\put(250,120){\line(-1,-1){30}}
\put(250,120){\line(1,-1){30}}
\put(310,60){\line(1,-1){30}}
\dottedline{4}(280,90)(310,60)
%%%%%%%%%% labelings

\put(285,90){$k_2$}
\put(215,80){0}
%%%%%%%%%%%%%%%  second level

%%%%%%%%%%%%%%%  first tree
\put(50,60){\circle*{4}}
\put(110,60){\circle*{4}}

\put(80,90){\line(-1,-1){30}}

%%%%%%%%%%%% labelings
\put(115,60){$w_{h-2}$}

%%%%%%%%%%%%%%secon tree
\put(250,60){\circle*{4}}
\put(310,60){\circle*{4}}

\put(280,90){\line(-1,-1){30}}

%%%%%%%%%%%% labelings
\put(315,60){$k_{h-1}$}
\put(245,50){0}

%%%%%%%%%%%%%%%%%%% third level

%%%%%%%%%%%%%%%%%%% first tree

\put(80,30){\circle*{4}}
\put(140,30){\circle*{4}}

\put(110,60){\line(-1,-1){30}}
%%%%%%%%%%%%%%%% labelings
\put(275,20){0}
\put(340,20){0}

%%%%%%%%%%%%%%%%%%% second tree

\put(280,30){\circle*{4}}
\put(340,30){\circle*{4}}

\put(310,60){\line(-1,-1){30}}

\end{picture}

We introduce a specific notation in this example: for $r=1,2,\dotsc,h-1$ we set

\[
\overline{a}_r=\min\{a_{r+1}+a_{r+2}+\dotsb+a_h,a_r+a_{r+2}+a_{r+3}+\dotsb+a_h,\dotsc,a_{r+1}+a_{r+2}+\dotsb+a_{h-1}\}.
\]

For $0\leq w_0\leq n$, $0\leq k_1\leq \min\{n-w_0,w_0,\overline{a}_1\}$, 

\[
\max\{0,k_{r-1}-a_{r-1}\}\leq k_r\leq\min\{a_r+\dotsb+a_h-k_{r-1},k_{r-1},\overline{a}_r\},\qquad r=2,3,\dotsc,h-2,
\]

and $\max\{0,k_{h-2}-a_{h-2}\}\leq k_{h-1}\leq \min\{a_{h-1}+a_h-k_{h-2},k_{h-2},\overline{a}_h\}$ (we also set $k_h=0$) the associated Hahn polynomials are 

\[
E_{\mathcal{T}_\text{\rm r}({\bf k})}(\mathcal{T}({\bf a}),w_0,\mathcal{T}_\text{\rm v}({\bf w}))=\prod_{r=1}^{h-1}
E_{k_r-k_{r+1}}(a_{r+1}+\dotsb+a_h-2k_{r+1},a_r,w_{r-1}-k_{r+1},w_r-k_{r+1})
\]

defined for 

\[
\max\{k_{r+1},w_{r-1}-a_r\}\leq w_r\leq\min\{a_{r+1}+\dotsb+a_h-k_{r+1},w_{r-1}\},\qquad r=1,2,\dotsc,h-1.
\]

The condition $k_{h-1}\leq \overline{a}_{h-1}$ comes form Corollary \ref{spherical} and forces $k_r\leq \overline{a}_r$, $r=h-2,h-3,\dotsc, 1$.

}
\end{example}

\begin{remark}
{\rm
The preceding example gives the multidimensional Hahn polynomials in \cite{K-M} and, as a special case, those in \cite{Du5}.

Indeed,
Dunkl's function  $(-1)^{r-m}\tilde{\theta}(u_1,u_2,u_3;a,b,c)$ is obtained by setting $h=3, (a_1,a_2,a_3)=(c,a,b),k_1=r,k_2=m$ and $m=M$, with the variables $w_1=u_1+u_2$ and $w_2=u_3$ (one has just to apply the symmetry relation (3.2) in \cite{Du2} (see \eqref{form5} in section \ref{secRegge} of the present paper) to $E_{r-m}(\dotsb)$, and this gives the factor $(-1)^{r-m}$). 

More generally,
in the Karlin and McGregor notation (\cite{K-M}, p.277) the function

\[
\phi\left(
\begin{array}{llllllll} w_{h-1}-w_h, &w_{h-2}-w_{h-1},&\dotsc,&w_0-w_1&&&&\\
&&&&\lvert k_{h-1},&k_{h-2}-k_{h-1},&\dotsc,& k_1-k_2\\
-a_{h-1}-1,&-a_{h-2}-1,&\dotsc,&-a_1-1&&&&
\end{array}	\right)
\]

is a multiple of our $E_{\mathcal{T}_\text{\rm r}({\bf k})}(\mathcal{T}({\bf a}),w_0,\mathcal{T}_\text{\rm v}({\bf w}))$ in the preceding example (one has to use the formulas in \cite{Du2}, p.631, keeping into account that $Q_m(x;\alpha,\beta,N)$ in \cite{Du2} is equal to $Q_m(x;\alpha,\beta,N+1)$ in \cite{K-M}). Tratnik \cite{Tratnik} gave a multi variable version of the Askey-Wilson polynomials that included the Karlin-McGregor multi variable Hahn polynomials (and the $q$-analog of Tratnik construction is given by Gaper and Rahman in \cite{GaRa}); it is natural to ask if the construction in the present paper may be generalized to the Tratnik (or Gasper-Rahman) setting (the $q$-analog of \cite{Du5} is in \cite{Du6}). Other multidimensional $q$-Hahn polynomials are in \cite{Rosengren}.
}
\end{remark}

%%%%%%%%%%%%%%%%%%%%%%%%%%%%%%%%%%%%%%%%%%%%%%%%%%%%%%%%%%%
%%%%%%%%%%%%%%%%%%%%%%%%%%%%%%%%%%%%%%%%%%%%%%%%%%%%%%%
\subsection{$S_{\bf a}- S_m\times S_{n-m}$ intertwining functions}
%%%%%%%%%%%%%%%%%%%%%%%%%%%%%%%%%%%%%%%%%%%%%%%%%%%%%%%%%%%
%%%%%%%%%%%%%%%%%%%%%%%%%%%%%%%%%%%%%%%%%%%%%%%%%%%%%%%

Now we define a set of $S_{\bf a}$-invariant functions in $M^{n-m,m}$. 

\begin{itemize}
\item

The function $\Phi(\mathcal{T}({\bf a}),m,\mathcal{T}_\text{\rm r}({\bf k}))\in M^{n-m,m}$ is defined by setting, iteratively,

\begin{multline}\label{Phifunctions}
\Phi(\mathcal{T}({\bf a}),m,\mathcal{T}_\text{\rm r}({\bf k}))=\sum_{w=\max\{j,i-\underline{a}_t+m\}}^{\min\{n-\underline{a}_t-j,m-i\}}E_{k-i-j}(n-\underline{a}_t-2j,\underline{a}_t-2i,m-i-j,w-j)
\times\\
\times\Phi(\mathcal{T}'({\bf a}),m-w,\mathcal{T}'_\text{\rm r}({\bf k}))\otimes\Phi(\mathcal{T}''({\bf a}),w,\mathcal{T}''_\text{\rm r}({\bf k})).
\end{multline}

\item If $\beta$ is a leaf (resp. $\gamma$ is a leaf) then $i=0$ and $t=1$ (resp. $j=0$ and $t=h-1$) and we set 
$\Phi(\mathcal{T}'({\bf a}),m-w,\mathcal{T}'_\text{\rm r}({\bf k}))=\sigma_{m-w}(A_1)$
(resp. $\Phi(\mathcal{T}''({\bf a}),w,\mathcal{T}''_\text{\rm r}({\bf k}))=\sigma_w(A_h)$).

\end{itemize}

In particular, if $\beta$ and $\gamma$ are both leaves we have
\begin{equation*}
\Phi(\mathcal{T}({\bf a}),m,\mathcal{T}_\text{\rm r}({\bf k}))=\sum_{w=\max\{0,m-a_1\}}^{\min\{n-a_1,m\}}E_k(a_2,a_1,m,w)
\sigma_{m-w}(A_1)\otimes\sigma_w(A_2).
\end{equation*}

\begin{lemma}\label{Fork}

For $k\leq r\leq n-k$, we have

\[
R_{r-m}\Phi(\mathcal{T}({\bf a}),m,\mathcal{T}_\text{\rm r}({\bf k}))=\binom{r-k}{m-k}\Phi(\mathcal{T}({\bf a}),r,\mathcal{T}_\text{\rm r}({\bf k}))
\]
\end{lemma}

\begin{proof}

The proof is by iteration. We limit ourselves to examine the case in which both $\beta$ and $\gamma$ are branch points.
Applying \eqref{RtRqRp}, with $t=r-m$, $p=q=0$ and $s=v-w$, we get

\begin{multline*}
R_{r-m}\Phi(\mathcal{T}({\bf a}),m,\mathcal{T}_\text{\rm r}({\bf k}))=
\sum_{w=\max\{j,i-\underline{a}_t+m\}}^{\min\{n-\underline{a}_t-j,m-i\}}\sum_{v=\max\{w,r-\underline{a}_t+i\}}^{\min\{n-\underline{a}_t-j,r-m+w\}}
E_{k-i-j}(n-\underline{a}_t-2j,\underline{a}_t-2i,m-i-j,w-j)
\times\\
\times\left[R_{r-m-v+w}\Phi(\mathcal{T}'({\bf a}),m-w,\mathcal{T}'_\text{\rm r}({\bf k}))\right]\otimes\left[R_{v-w}\Phi(\mathcal{T}''({\bf a}),w,\mathcal{T}''_\text{\rm r}({\bf k}))\right]\\
=\sum_{v=\max\{j,r-\underline{a}_t+i\}}^{\min\{n-\underline{a}_t-j,r-i\}}\Biggl[\sum_{w=\max\{j,v-r+m\}}^{\min\{v,m-i\}} E_{k-i-j}(n-\underline{a}_t,\underline{a}_t-2i,m-i-j,w-j)\times\\
\times\binom{r-v-i}{m-w-i}\binom{v-j}{w-j}\Biggr]
\Phi(\mathcal{T}'({\bf a}),r-w-s,\mathcal{T}'_\text{\rm r}({\bf k}))\otimes \Phi(\mathcal{T}''({\bf a}),w+s,\mathcal{T}''_\text{\rm r}({\bf k}))\\
=\binom{r-k}{m-k}\Phi(\mathcal{T}({\bf a}),r,\mathcal{T}_\text{\rm r}({\bf k})),
\end{multline*}

where the second equality follows from the induction hypothesis and last from \eqref{form7}.

\end{proof}

\begin{corollary}\label{Inthenot}
In the notation of Theorem \ref{LRexplicit}, we have:

\[
\Phi(\mathcal{T}({\bf a})),m,\mathcal{T}_\text{\rm r}({\bf k}))=T_m[\Phi(\mathcal{T}'({\bf a}),i,\mathcal{T}'_\text{\rm r}({\bf k}))\otimes \Phi(\mathcal{T}''({\bf a}),j,\mathcal{T}''_\text{\rm r}({\bf k}))]
\]

\end{corollary}

Let $L$ be the height of $\mathcal{T}$. To the labeled tree $\mathcal{T}({\bf a})$, we associate a chain of subgroups $K_0= S_n\geq K_1\geq\dotsb\geq K_{L-1}\geq K_L=S_{\bf a}$ defined as follows. Let $L'$ and $L''$ be respectively the height of $\mathcal{T}'$ and $\mathcal{T}''$ and suppose that $K'_0= S_{\underline{a}_t}\geq K'_1\geq\dotsb\geq K'_{L'}=S_{a_1}\times \dotsb\times S_{a_t}$ and 
$K''_0= S_{n-\underline{a}_t}\geq K''_1\geq\dotsb\geq K''_{L''}=S_{a_{t+1}}\times \dotsb\times S_{a_h}$ are the respective chains. Clearly, $L-1=\max\{L',L''\}$ and if $L'<L-1$ (or $L''<L-1$) we set $K'_{L'+1}=\dotsb =K'_{L-1}:=K'_{L'}$ (resp. $K''_{L''+1}=\dotsb =K''_{L-1}:=K''_{L''}$). Then
we define the chain of $\mathcal{T}({\bf a})$ by setting

\[
K_0=S_n
\qquad\text{and}\qquad K_j=K'_{j-1}\times K''_{j-1},\quad\text{for}\quad j=1,2,\dotsc, L.
\]

For instance, the chain associated to Example \ref{h4} is $S_n\geq S_{a_1+a_2}\times S_{a_3 + a_4}\geq S_{a_1}\times S_{a_2}\times S_{a_3}\times S_{a_4}$, while the chain associated to Example \ref{ah} is $S_n\geq S_{a_1}\times S_{a_2+\dotsb+a_h}\geq S_{a_1}\times S_{a_2}\times S_{a_3+\dotsb+a_h}\geq \dotsb\geq S_{a_1}\times S_{a_2}\times\dotsb\times S_{a_h}$.

Let $i,j,k$ be as in \eqref{ijk}.
We define a chain of irreducible representations $W_0= S^{n-k,k}\supseteq W_1\supseteq\dotsb\supseteq W_L= S^{(a_1)}\otimes\dotsb\otimes S^{(a_h)}$ associated to the labelings $\mathcal{T}({\bf a})$ and $\mathcal{T}_\text{r}({\bf k})$; 
the subspace $W_j$ is an irreducible representation of $K_j, j=0,1,2,\dotsc ,L$. Suppose that $W'_0= S^{\underline{a}_1-i,i}\supseteq W'_1\supseteq \dotsb\supseteq W'_{L-1}$ and $W''_0= S^{n-\underline{a}_1-j,j}\supseteq W''_1\supseteq \dotsb\supseteq W''_{L-1}$ are the chains associated to $\mathcal{T}'$ and $\mathcal{T}''$. Then the chain associated to $\mathcal{T}$ is 

\begin{equation}\label{Wchain}
W_0= S^{n-k,k}\qquad\text{and}\qquad
W_j=W'_{j-1}\otimes W''_{j-1},\quad j=1,2,\dotsc,L.
\end{equation}

Since the decomposition in \eqref{LR} is multiplicity free, the chain of representations is uniquely determined by the tree and its labelings.
In other words, the restriction $\text{Res}^{S_n}_{S_{\bf a}}S^{n-k,k}$ does not decompose multiplicity free, but using the transitivity of restriction, we may write

\begin{equation}\label{Res}
\text{Res}^{S_n}_{S_{\bf a}}S^{n-k,k}=\text{Res}^{K_{L-1}}_{K_L}\dotsb\text{Res}^{K_1}_{K_2}\text{Res}^{K_0}_{K_1}S^{n-k,k}.
\end{equation}

Then $\text{Res}^{K_0}_{K_1}S^{n-k,k}$ decomposes multiplicity free and 
we choose the unique subspace isomorphic to $S^{\underline{a}_1-i,i}\otimes S^{n-\underline{a}_1-j,j}$, and we iterate this procedure. Since the leaves of $\mathcal{T}_\text{r}({\bf k})$ are labeled with zeros, this implies that $W_L$ is the trivial representation  of $S_{a_1}\times \dotsb\times S_{a_h}$.

For instance, the chain of representations associated to Example \ref{h4} is $S^{n-k,k}\supseteq S^{a_1+a_2-i,i}\otimes S^{a_3+a_4-j,j}\geq S^{(a_1)}\otimes S^{(a_2)}\otimes S^{(a_3)}\otimes S^{(a_4)}$, while the chain associated to Example \ref{ah} is $S^{n-k_1,k_1}\supseteq S^{(a_1)}\otimes S^{a_2+\dotsb+a_h-k_2,k_2}\supseteq S^{(a_1)}\otimes S^{(a_2)}\otimes S^{a_3+\dotsb+a_h-k_3,k_3}\supseteq \dotsb\supseteq S^{(a_1)}\otimes S^{(a_2)}\otimes\dotsb\otimes S^{(a_h)}$.

Denote by $\left[S^{(a_1)}\otimes \dotsb\otimes S^{(a_h)}\right]_{\mathcal{T}_\text{r}({\bf k})}$ the subspace $W_L$ obtained by mean of the particular labeling $\mathcal{T}_\text{r}({\bf k})$ (while $\mathcal{T}({\bf a})$ is constant). Then

\begin{equation}\label{dec}
\bigoplus_{\substack{\mathcal{T}_\text{r}({\bf k})\\k\text{ \rm fixed}}}\left[S^{(a_1)}\otimes \dotsb\otimes S^{(a_h)}\right]_{\mathcal{T}_\text{r}({\bf k})}
\end{equation}

is an orthogonal decomposition of the $S^{(a_1)}\otimes \dotsb\otimes S^{(a_h)}$-isotypic component of $\text{Res}^{S_n}_{S_{a_1}\times \dotsb\times S_{a_h}}S^{n-k,k}$ (the sum is over all labelings of $\mathcal{T}$ with $k$ as label of the root). Indeed, different choices of the parameters of the labeling $\mathcal{T}_\text{r}({\bf k})$ give rise to orthogonal subspaces (at the first stage of \eqref{Res} in which they differ); and varying the parameters we obtain a every possible subspace in the intermediate decompositions. 

\begin{theorem}\label{belsub}
The $S_{\bf a}-(S_h\times S_{n-h})$-invariant function $\Phi(\mathcal{T}({\bf a}),m,\mathcal{T}_\text{\rm r}({\bf k}))$ belongs to the subspace $R_{m-k}\left[S^{(a_1)}\otimes \dotsb\otimes S^{(a_h)}\right]_{\mathcal{T}_\text{r}({\bf k})}$ contained in $R_{m-k}S^{n-k,k}\subseteq M^{n-m,m}$.
\end{theorem} 

\begin{proof}
The proof is again by iteration. Suppose that it is true for $\Phi(\mathcal{T}'({\bf a}),m-w,\mathcal{T}'_\text{\rm r}({\bf k}))$ and $\Phi(\mathcal{T}''({\bf a}),w,\mathcal{T}'_\text{\rm r}({\bf k}))$. Then one can use Lemma \ref{Fork} and Theorem \ref{LRexplicit} to deduce that $\Phi(\mathcal{T}({\bf a}),m,\mathcal{T}_\text{\rm r}({\bf k}))$ belongs to the subspace $S^{h-i,i}\otimes S^{n-h-j,j}$ contained in the restriction of $R_{m-k}S^{n-k,k}$ to $S_{n-\underline{a}_t}\times S_{\underline{a}_t}$. 
\end{proof}

\begin{corollary}
The set $\{\Phi(\mathcal{T}({\bf a}),m,\mathcal{T}_\text{\rm r}({\bf k})):\mathcal{T}_\text{\rm r}({\bf k})\;\text{\rm is a rep. labeling }\}$ is an orthogonal basis for the space of $S_{\bf a}$-invariant functions in $M^{n-m,m}$.
\end{corollary}

The norm of $\Phi(\mathcal{T}({\bf a}),m,\mathcal{T}_\text{\rm r}({\bf k}))$ will be computed in the next section.

%%%%%%%%%%%%%%%%%%%%%%%%%%%%%%%%%%%%%%%%%%%%%%%%%%%%%%%%%%%%%%%%%%%
%%%%%%%%%%%%%%%%%%%%%%%%%%%%%%%%%%%%%%%%%%%%%%%%%%%%%%%%%%%%%%%%%%
\subsection{The $S_{\bf a}-S_m\times S_{n-m}$-intertwining functions as Hahn polynomials}\label{secPhiHahn}
%%%%%%%%%%%%%%%%%%%%%%%%%%%%%%%%%%%%%%%%%%%%%%%%%%%%%%%%%%%%%%%%
%%%%%%%%%%%%%%%%%%%%%%%%%%%%%%%%%%%%%%%%%%%%%%%%%%%%%%%%%%%%%%%%%

Now we want to express the functions $\Phi(\mathcal{T}({\bf a}),m,\mathcal{T}_\text{\rm r}({\bf k}))$ in terms of the Hahn polynomials $E_{\mathcal{T}_\text{\rm r}({\bf k})}(\mathcal{T}({\bf a}),m,\mathcal{T}_\text{\rm v}({\bf w}))$. 
Set $A_j=\{a_1+\dotsb+a_{j-1}+1,\dotsc,a_1+\dotsb+a_j\}$ and suppose that $S_{a_j}$ is the stabilizer of $A_j$, $j=1,2,\dotsc, h$. If $x_1,x_2,\dotsc, x_h$ are nonnegative integers satisfying $x_1+x_2+\dotsb +x_h=m$, we denote by 
$\sigma_{x_1}(A_1)\otimes\sigma_{x_2}\otimes(A_2)\otimes\dotsb\otimes\sigma_{x_h}(A_h)$ the characteristic function of the set of all $B\in\Omega_m$ such that $\lvert B\cap A_j\rvert=x_j, j=1,2,\dotsc,h$; compare with the notation in Corollary \ref{spherical}.  Since the functions 
$\sigma_{x_1}(A_1)\otimes\sigma_{x_2}(A_2)\otimes\dotsb\otimes\sigma_{x_h}(A_h)$ form a basis for the $S_{\bf a}$-invariant functions in $M^{n-m,m}$, each $\Phi(\mathcal{T}({\bf a}),m,\mathcal{T}_\text{\rm r}({\bf k}))$ must be a linear combination of these functions. The following Lemma is obvious.

\begin{lemma}\label{PhiHahn}
We have 

\[
\Phi(\mathcal{T}({\bf a}),m,\mathcal{T}_\text{\rm r}({\bf k}))=\sum_{\mathcal{T}_\text{\rm v}({\bf w})}
E_{\mathcal{T}_\text{\rm r}({\bf k})}(\mathcal{T}({\bf a}),m,\mathcal{T}_\text{\rm v}({\bf w}))\sigma_{x_1}(A_1)\otimes\sigma_{x_2}(A_2)\otimes\dotsb\otimes\sigma_{x_h}(A_h)
\]

where the sum is over all possible value of the variables in ${\mathcal{T}_\text{\rm v}({\bf w})}$ and the coefficients $x_1,\dotsc, x_h$ must have the following values:

\begin{itemize}
\item if $a_j$ is the label of a leaf which is the left son of a vertex with $y$ as  space label and $z$ as variable label, then $x_j=y-z$;
\item if $a_j$ is the label of a leaf which is the right son of a vertex with $z$ as variable label, then $x_j=z$.
\end{itemize}
\end{lemma}

For instance, in Example \ref{h4} the invariant functions have the expression:

\begin{multline*}
\sum_{w,u,v} E_{k-i-j}(a_3+a_4-2j,a_1+a_2-2i,m-i-j,w-j)
 E_i(a_2,a_1,m-w,v)\times\\
\times E_j(a_4,a_3,w,u) 
\sigma_{m-w-v}(A_1)\otimes\sigma_v(A_2)\otimes\sigma_{w-u}(A_3)\otimes\sigma_{u}(A_4),
\end{multline*}

while in Example \ref{ah} they are

\begin{multline*}
\sum_{w_1,\dotsc,w_{h-1}}\left(\prod_{j=1}^{h-1}
E_{k_j-k_{j+1}}(a_{j+1}+\dotsb+a_h-2k_{j+1},a_j,w_{j-1}-k_{j+1},w_j-k_{j+1})\right)\sigma_{w_0-w_1}(A_1)\otimes\\
\otimes\sigma_{w_1-w_2}(A_2)\otimes\dotsb\otimes\sigma_{w_{h-2}-w_{h-1}}(A_{h-2})\otimes\sigma_{w_{h-1}}(A_h)
\end{multline*}

The polynomials $E_{\mathcal{T}_\text{\rm r}({\bf k})}(\mathcal{T}({\bf a}),m,\mathcal{T}_\text{\rm v}({\bf w}))$ satisfy a number of properties that may be deduced from their group theoretic interpretation or from the identities in section \ref{df} (and those in \cite{Du1}).
We give two examples.

Suppose that $\mathcal{T}$ is a tree, with the labelings used in the previous sections. Suppose that $\mathcal{T}_\text{v}({\bf v})$ is another variable labeling and $\mathcal{T}_\text{s}(r)$ another space labeling of the same tree. We define a set of binomial coefficients by setting

\[
\binom{r,\mathcal{T}_\text{v}({\bf v})}{m,\mathcal{T}_\text{v}({\bf w})}=
\binom{r-v,\mathcal{T}'_\text{v}({\bf v})}{m-w,\mathcal{T}_\text{v}'({\bf w})}
\binom{v,\mathcal{T}''_\text{v}({\bf v})}{w,\mathcal{T}''_\text{v}({\bf w})}\qquad \text{if }\beta\text{ and }\gamma\text{ are internal;}
\] 

if $\beta$ is a leaf, $\binom{r-v,\mathcal{T}'_\text{v}({\bf v})}{m-w,\mathcal{T}_\text{v}'({\bf w})}$ is replaced by $\binom{r-v}{m-w}$; if $\gamma$ is a leaf, $\binom{v,\mathcal{T}''_\text{v}({\bf v})}{w,\mathcal{T}''_\text{v}({\bf w})}$ is replaced by $\binom{v}{w}$.

Then from Lemma \ref{Fork} (alternatively, from \eqref{form7}) one can easily deduce the multiplication formula:

\[
\sum_{\mathcal{T}_\text{\rm v}({\bf w})}E_{\mathcal{T}_\text{\rm r}({\bf k})}(\mathcal{T}({\bf a}),m,\mathcal{T}_\text{\rm v}({\bf w}))
\binom{r,\mathcal{T}_\text{\rm v}({\bf v})}{m,\mathcal{T}_\text{\rm v}({\bf w})}
=\binom{r-k}{m-k}E_{\mathcal{T}_\text{\rm r}({\bf k})}(\mathcal{T}({\bf a}),r,\mathcal{T}_\text{\rm v}({\bf v})).
\]

Now suppose that $\mathcal{T}_\text{\rm r}({\bf k})$ and $\mathcal{T}_\text{\rm r}({\bf \overline{k}})$ are two representation labelings for $\mathcal{T}$. Clearly,

\[
\left\langle\Phi(\mathcal{T}({\bf a}),m,\mathcal{T}_\text{\rm r}({\bf k})),\Phi(\mathcal{T}({\bf a}),m,\mathcal{T}_\text{\rm r}({\bf \overline{k}}))\right\rangle_{M^{n-m,m}}=\delta_{\mathcal{T}_\text{\rm r}({\bf k}),\mathcal{T}_\text{\rm r}({\bf \overline{k}})}\lVert \Phi(\mathcal{T}({\bf a}),m,\mathcal{T}_\text{\rm r}({\bf k}))\rVert_{M^{n-m,m}}^2.
\]

Moreover, the norm $\lVert \Phi(\mathcal{T}({\bf a}),m,\mathcal{T}_\text{\rm r}({\bf k}))\rVert_{M^{n-m,m}}^2$ may be easily computed by iteration using \eqref{norm} and Corollary \ref{Inthenot}:

\[
\lVert \Phi(\mathcal{T}({\bf a}),m,\mathcal{T}_\text{\rm r}({\bf k}))\rVert_{M^{n-m,m}}^2=\binom{n-2i-2j}{m-i-j}\binom{n-2i-2j}{k-i-j}^{-1}\cdot\prod_{\omega}l_\omega
\]

where the product is over all internal points $\omega$ of $\mathcal{T}$ and 

\[
\begin{split}
l_\omega=&\frac{a_r+\dotsb+a_t-x-y-z+1}{a_r+\dotsb+a_t-2x+1}(a_{s+1}+\dotsb+a_t-z-x+y+1)_{x-y-z}\times\\
\times&(a_r+\dotsb+a_s-y-x+z+1)_{x-y-z}
(q-x+1)_{x-y-z}(a_r+\dotsb+a_t-q-x+1)_{x-y-z}
\end{split}
\]

 if
$(a_r,a_{r+1},\dotsc,a_t)$ is the composition label of $\omega$, $q$ and $x$ respectively its space and representation label, $(a_r,\dotsc,a_s)$ and $(a_{s+1},\dotsc,a_t,)$ the composition labels of his sons (left and right) and $y,z$ their representation labels.
Define a weight by setting, inductively, 

\begin{multline*}
W(\mathcal{T}({\bf a}),m,\mathcal{T}_\text{\rm r}({\bf k}),\mathcal{T}_\text{\rm v}({\bf w}))=\binom{\underline{a}_t-2i}{m-w-i}\binom{n-\underline{a}_t-2j}{w-j}\times\\
\times W(\mathcal{T}'({\bf a}),m-w,\mathcal{T}'_\text{\rm r}({\bf k}),\mathcal{T}'_\text{\rm v}({\bf w}))W(\mathcal{T}''({\bf a}),w,\mathcal{T}''_\text{\rm r}({\bf k}),\mathcal{T}''_\text{\rm v}({\bf w}))
\end{multline*}

and $W(\mathcal{T}'({\bf a}),m-w,\mathcal{T}'_\text{\rm r}({\bf k}),\mathcal{T}'_\text{\rm v}({\bf w}))=1$ (resp. $W(\mathcal{T}''({\bf a}),w,\mathcal{T}''_\text{\rm r}({\bf k}),\mathcal{T}''_\text{\rm v}({\bf w}))=1$) if $\beta$ is a leaf (resp. $\gamma$ is a leaf). Then the orthogonality relations are:

\begin{multline*}
\sum_{\mathcal{T}_\text{\rm v}({\bf w})}E_{\mathcal{T}_\text{\rm r}({\bf k})}(\mathcal{T}({\bf a}),m,\mathcal{T}_\text{\rm v}({\bf w}))E_{\mathcal{T}_\text{\rm r}({\bf \overline{k}})}(\mathcal{T}({\bf a}),m,\mathcal{T}_\text{\rm v}({\bf w}))W(\mathcal{T}({\bf a}),m,\mathcal{T}_\text{\rm r}({\bf k}),\mathcal{T}_\text{\rm v}({\bf w}))\\
=\delta_{\mathcal{T}_\text{\rm r}({\bf k}),\mathcal{T}_\text{\rm r}({\bf \overline{k}})}\lVert \Phi(\mathcal{T}({\bf a}),m,\mathcal{T}_\text{\rm r}({\bf k}))\rVert_{M^{n-m,m}}^2
\end{multline*}

%%%%%%%%%%%%%%%%%%%%%%%%%%%%%%%%%%%%%%%%%%%%%%%%%%%%%%%%%%%%%%%%%%%%%%%%%%%%%%%%
%%%%%%%%%%%%%%%%%%%%%%%%%%%%%%%%%%%%%%%%%%%%%%%%%%%%%%%%%%%%%%%%%%%%%%%%%%%%%%%%
\subsection{On Frobenius reciprocity and Mackey's Lemma for a permutation representation}
%%%%%%%%%%%%%%%%%%%%%%%%%%%%%%%%%%%%%%%%%%%%%%%%%%%%%%%%%%%%%%%%%%%%%%%%%%%%%%%%
%%%%%%%%%%%%%%%%%%%%%%%%%%%%%%%%%%%%%%%%%%%%%%%%%%%%%%%%%%%%%%%%%%%%%%%%%%%%%%%%
In this subsection, we establish some technical facts on permutation and induced representations. We begin with a closer look at Frobenius reciprocity. Let $G$ be a finite group and $K$ a subgroup of $G$.
If $V$ is any $G-$representation,  we denote by
$V^K$ the subspace of $K-$invariant vectors in $V$. If $X=G/K$ and $x_0\in X$ is the point stabilized by $K$, then the space $L(X)$ may be identified with the space of right $K$-invariant functions on $G$ via the map

\begin{equation}\label{ftilde}
\begin{split}
L(X)&\longrightarrow L(G)\\
f&\longmapsto \widetilde{f}
\end{split}
\end{equation}

where $\widetilde{f}(g)=f(gx_0)$ for any $g\in G$.\\

Suppose now that $(\rho,V)$ is irreducible and unitary and that $V^K$ is non--trivial. 
For any $v \in V^K$, define the linear map
 $T_v : V \to L(X)$ by setting

\begin{equation}\label{proiezione}
(T_v w)(g x_0) =  \langle w, \rho(g)v\rangle_V,
\end{equation}

for any $w \in V$ and $g \in G$. Clearly, $T_vw\in L(X)$ is well defined because $v$ is $K$-invariant. Moreover, it is easy to see that $T_v\in\text{Hom}_G(V,L(X))$.

In particular, one can easily gets Frobenius reciprocity for permutation representations: the map 

\begin{equation}\label{Frobenius}
\begin{split}
V^K &\longrightarrow \text{\rm Hom}_G(V,L(X))\\
v &\longmapsto T_v
\end{split}
\end{equation}

is an anti linear isomorphism between the vector spaces
$V^K$ and $\text{\rm Hom}_G(V,L(X))$ (and therefore the multiplicity of $V$ in $L(X)$ is
equal to the dimension of $V^K$ \cite{CST1}).\\

In the following Lemma, we compute the invariant vector associated to the restriction of the intertwining operator $R_{m-k}$ to $S^{n-k,k}$; see Theorem \ref{corradtr}. We use Corollary \ref{spherical}: $\Phi(n,m,k,k)$ is the $S_{n-m}\times S_m$-invariant vector in $S^{n-k,k}$. 

\begin{lemma}\label{RmkSnk}
In the anti linear isomorphism \eqref{Frobenius}, we have:

$R_{m-k}|_{S^{n-k,k}}=T_v$, where 

\[
v=\frac{(-1)^k}{(n-2k+2)_{k-1}k!}\Phi(n,m,k,k).
\]

\end{lemma}
\begin{proof}

Set $Y=\{1,2,\dotsc,m\}$, $Z=\{m+1,m+2,\dotsc,n\}$ and suppose that $S_{n-m}\times S_m$ is the stabilizer of $Y$.
First of all, we prove that on the whole $M^{n-k,k}$ we have $R_{m-k}=T_u$, with $u=\sigma_k(Y)\otimes\sigma_0(Z)$.
Indeed, if $F\in\Omega_k$ then \eqref{proiezione} yields

\[
\begin{split}
(T_u\delta_F)(gY)&=\langle\delta_F,g\left[\sigma_k(Y)\otimes\sigma_0(Z)\right]\rangle=\left\{
\begin{array}{l}
1\text{ if }F\subseteq gY\\
0\text{ otherwise}
\end{array}\right.\\
&=(R_{m-k}\delta_F)(gY)
\end{split}
\]

for every $g\in S_n$.
Then it suffices to compute the projection of $\sigma_k(Y)\otimes\sigma_0(Z)$ onto $S^{n-k,k}$: since the dimension of the $S_{n-m}\times S_m$-invariant vectors is equal to 1 (and is spanned by $\Phi(n,m,k,k)$), such projection is given by 

\begin{multline*}
\frac{\Phi(n,m,k,k)}{\lVert\Phi(n,m,k,k)\rVert^2_{S^{n-k,k}}}
\langle \sigma_k(Y)\otimes\sigma_0(Z),\Phi(n,m,k,k)\rangle_{M^{n-k,k}}\\
=E_k(n-m,m,k,0)
\frac{\lVert\sigma_k(Y)\otimes\sigma_0(Z)\rVert^2_{M^{n-k,k}}}{\lVert\Phi(n,m,k,k)\rVert^2_{S^{n-k,k}}}\Phi(n,m,k,k)
=\frac{(-1)^k}{(n-2k+2)_k k!}\Phi(n,m,k,k).
\end{multline*}

\end{proof}

By combining the preceding Theorem with Lemma \ref{Fork}, we get the expression of 
$\Phi(\mathcal{T}({\bf a}),m,\mathcal{T}_\text{\rm r}({\bf k}))$ as a matrix coefficient. 

\begin{corollary}\label{etatilde}
If $\eta=\Phi(\mathcal{T}({\bf a}),m,\mathcal{T}_\text{\rm r}({\bf k}))$ then

\[
\widetilde{\eta}(g)=\frac{(-1)^k}{(n-2k+2)_kk!}\langle\Phi(\mathcal{T}({\bf a}),k,\mathcal{T}_\text{\rm r}({\bf k})),g\Phi(n,m,k,k)\rangle_{S^{n-k,k}}.
\]

\end{corollary}

Now we examine the behavior of the map  \eqref{proiezione} with respect to transitivity of induction.
Suppose that that $H$ is another group with  $K\leq H\leq G$. Set $Y=G/H$ and $Z=H/K$, and let $y_0\in Y$ be the point stabilized by $H$ in $Y$. Identify $Z$ with the subset $\{hx_0:h\in H\}$ of $X$, so that $x_0$ is also the point stabilized by $K$ in $Z$. For every $y\in Y$ choose $s_y\in G$ such that $s_yy_0=y$. Then we have $G=\coprod\limits_{y\in Y}s_yH$, that is $\{s_y:y\in Y\}$ is a transversal for $H$ in $G$; in particular

\begin{equation}\label{XyY}
X=\coprod_{y\in Y}s_y Z.
\end{equation}

Transitivity of induction for a permutation representation may be made more explicit by saying that 

\begin{equation}\label{LX}
L(X)\equiv\text{Ind}_H^G L(Z)=\bigoplus\limits_{y\in Y}L(s_yZ),
\end{equation}

as follows immediately from \eqref{XyY}.
Let $(\rho,V)$ be again a unitary, irreducible
representation of $G$ and $W$ an $H$-invariant, irreducible subspace
of $V$. That is, $(\sigma,W)$ is an irreducible representation in
$\text{Res}^G_HV$ and $\sigma(h)w=\rho(h)w$ for all $w\in W$ and $h\in H$. 
If $w_0\in W$ is $K$-invariant, then, by mean of
\eqref{proiezione}, we can form two distinct intertwining operators:
$S_{w_0}:W\rightarrow L(Z)$ (that commutes with $H$) and
$T_{w_0}:V\rightarrow L(X)$ (that commutes with $G$). From the definition of induced representation \cite{Serre} and \eqref{LX},
$\text{Ind}_H^GS_{w_0}W$ is a well defined subspace of $L(X)$, namely

\begin{equation}\label{IndHG}
\text{Ind}_H^GS_{w_0}W=\bigoplus_{y\in Y}s_yS_{w_0}W
\end{equation}

\begin{lemma}\label{traind}
The operator $T_{w_0}$ intertwines $V$ with the subspace $\text{\rm Ind}_H^GS_{w_0}W$, that is $T_{w_0}V\subseteq\text{\rm Ind}_H^GS_{w_0}W$.
\end{lemma}

\begin{proof}
Denote by $P_W:V\rightarrow W$ the orthogonal projection onto $W$.
Suppose that $x\in X$ and $x=s_yhx_0$, with $y\in Y$ and $h\in H$. By \eqref{proiezione}, for any $v\in V$ we have

\[
\begin{split}
(T_{w_0}v)(x)=&(T_{w_0}v)(s_yhx_0)=\langle v,\rho(s_y)\rho(h)w_0\rangle_V\\
=&\langle P_W\rho(s_y^{-1}) v,\sigma(h)w_0\rangle_V
=\left[S_{w_0}P_W\rho(s_y^{-1})v\right](hx_0)\\
=&\left\{s_y\left[S_{w_0}P_W\rho(s_y^{-1})v\right]\right\}(x)\qquad(\text{because }x=s_yhx_0),
\end{split}
\]

that is, $T_{w_0}v\in s_yS_{w_0}W$, and by \eqref{IndHG} this shows that $T_{w_0}v\in\text{Ind}_H^GS_{w_0}W$.
\end{proof}

In the last part of this subsection, we show how to apply Mackey's lemma in our setting.
Assume all the preceding notation and suppose that $J$ is another subgroup of $G$. 
Suppose that $T$ is a set of representatives for the $J-H$ double cosets of $G$, that is $G=\coprod_{t\in T}JtH$. Set $\Omega_t=\{gty_0:g\in J\}$. Then $Y=\coprod_{t\in T}\Omega_t$  is the decompositions of $Y$ into $J$ orbits.
Let $J_t$ be the stabilizer of $ty_0\in\Omega_t$ in $J$. Clearly, $J_t$ is also the group of all $g\in J$ such that $gtZ=tZ$.
Then Mackey's Lemma \cite{Serre} in our context coincides with:

\begin{equation}\label{Mackey}
\text{Res}^G_JL(X)\equiv\text{Res}^G_J\text{Ind}^G_H L(Z)=\bigoplus_{t\in T}\text{Ind}_{J_t}^JL(tZ)\equiv\bigoplus_{t\in T}L(\Theta_t)
\end{equation}

where $\Theta_t=\{gtz:z\in Z,g\in J\}$.

%%%%%%%%%%%%%%%%%%%%%%%%%%%%%%%%%%%%%%%%%%%%%%%%%%%%%%%%%%%%%%%%%%%%%%%%%%%%%%%%
%%%%%%%%%%%%%%%%%%%%%%%%%%%%%%%%%%%%%%%%%%%%%%%%%%%%%%%%%%%%%%%%%%%%%%%%%%%%%%%%
\subsection{The $S_{n-m}\times S_m$-invariant functions in $M^{\bf a}$}
%%%%%%%%%%%%%%%%%%%%%%%%%%%%%%%%%%%%%%%%%%%%%%%%%%%%%%%%%%%%%%%%%%%%%%%%%%%%%%%%
%%%%%%%%%%%%%%%%%%%%%%%%%%%%%%%%%%%%%%%%%%%%%%%%%%%%%%%%%%%%%%%%%%%%%%%%%%%%%%%%

Let ${\bf a}=(a_1,\dotsc,a_h)$ be again a fixed composition of $n$. We denote by $\Omega_{\bf a}$ the $S_n$-homogeneous space of all $(B_1,B_2,\dotsc,B_h)$ such that:
$B_j$ is a $j$-subset of $\{1,2,\dotsc,n\}$ and $B_i\cap B_j=\emptyset$ for $i\neq j$, (and therefore $B_1\cup\dotsb\cup B_h=\{1,2,\dotsc,n\}$). 
Following the standard notation, we denote by $M^{\bf a}$ the $S_n$-permutation module $L(\Omega_{\bf a})$.
If ${\bf x}=(x_1,\dotsc,x_h)$ is another $h$-parts composition, we will write ${\bf x}\leq {\bf a}$ when $x_j\leq a_j, j=1,2,\dotsc,h$. If ${\bf x}\leq {\bf a}$, ${\bf x}\Vdash l\leq n$ and $F$ is an $l$-subset of $\{1,2,\dotsc,n\}$, we can construct an $\Omega_{\bf x}$ space using the subsets of $F$, and we denote it by $\Omega_{\bf x}(F)$; the corresponding permutation module will be denoted by $M^{\bf x}(F)$. 

Set again 
$Y=\{1,2,\dotsc,m\}$ and $Z=\{m+1,m+2,\dotsc,n\}$ and suppose that $S_{n-m}\times S_m$ is the stabilizer of $Y\in\Omega_m$. 
The decomposition of $\Omega_{\bf a}$ into $S_{n-m}\times S_m$ orbits is:

\begin{equation}\label{Omegaxa}
\Omega_{\bf a}=\coprod_{\substack{{\bf x}\Vdash m\\ {\bf x}\leq {\bf a}}}\left[\Omega_{\bf x}(Y)\times \Omega_{{\bf a}-{\bf x}}(Z)\right],
\end{equation}

where ${\bf a}-{\bf x}=(a_1-x_1,\dotsc,a_h-x_h)$ and $\Omega_{\bf x}(Y)\times \Omega_{{\bf a}-{\bf x}}(Z)\equiv \{(B_1,B_2,\dotsc,B_h)\in\Omega_{\bf a}:\lvert B_j\cap Y\rvert=x_j,j=1,2,\dotsc,h\}$.
The consequent decomposition of $\text{Res}^{S_n}_{S_{n-m}\times S_m}M^{\bf a}$ into permutation modules is

\[
M^{\bf a}=\bigoplus_{\substack{{\bf x}\Vdash m\\ {\bf x}\leq {\bf a}}}\left[M^{\bf x}(Y)\otimes M^{{\bf a}-{\bf x}}(Z)\right].
\]

For ${\bf x}\Vdash m$, ${\bf x}\leq{\bf a}$, we denote by $\sigma_{\bf x}(Y)$ the function in $M^{\bf x}(Y)$ which is constant and identically 1. Then the tensor product $\sigma_{\bf x}(Y)\otimes \sigma_{{\bf a}-{\bf x}}(Z)$ is the characteristic function of the orbit $\Omega_{\bf x}(Y)\times \Omega_{{\bf a}-{\bf x}}(Z)$.\\

If $\eta\in M^{n-m,m}$, then \eqref{ftilde} yields the function defined by setting $\widetilde{\eta}(g)=\eta(gY)$; similarly, if $\theta\in M^{\bf a}$, then  $\widetilde{\theta}(g)=\theta(gA_1,\dotsc, gA_h)$. If $\eta\in M^{n-m,m}$ is $S_{\bf a}$-invariant, then $\widetilde{\eta}$ is also left $S_{\bf a}-(S_{n-m}\times S_m)$-invariant; then we can define $\eta^\sharp\in M^{\bf a}$ by requiring that $\widetilde{\eta}(g^{-1})=\widetilde{\eta^\sharp}(g)$. In other words, the change of variable $g\mapsto g^{-1}$ establishes an isomorphism between the vector spaces of $S_{\bf a}-(S_{n-m}\times S_m)$-invariant functions and $(S_{n-m}\times S_m)-S_{\bf a}$-invariant functions on $S_n$, and on the homogeneous spaces this isomorphism becomes $(M^{n-m,m})^{S_{\bf a}}\rightarrow (M^{\bf a})^{S_{n-m}\times S_m}$, $\eta\mapsto\eta^\sharp$. \\

If $\eta=\sigma_{x_1}(A_1)\otimes\dotsb\otimes\sigma_{x_h}(A_h)$ and $\theta=\sigma_{x_1,\dotsc,x_h}(Y)\otimes \sigma_{a_1-x_1,\dotsc,a_h-x_h}(Z)$ then $\eta^\sharp =\theta$. Indeed, $\widetilde{\eta}$ is the characteristic function of all $g\in S_n$ such that $\lvert gY\cap A_j\rvert=x_j, j=1,2,\dotsc,h$, while $\widetilde{\theta}$ is the characteristic function of all $g\in S_n$ such that $\lvert Y\cap gA_j\rvert=x_j, j=1,2,\dotsc,h$, and therefore $\widetilde{\eta}(g^{-1})=\widetilde{\theta}(g)$.
As an immediate consequence, we get:

\begin{lemma}\label{LemmaPsifun}
Under the map $\eta\mapsto\eta^\sharp$, the image of $\Phi(\mathcal{T}({\bf a}),m,\mathcal{T}_\text{\rm r}({\bf k}))$ is the function

\begin{equation}\label{Psifunctions}
\Psi(\mathcal{T}({\bf a}),m,\mathcal{T}_\text{\rm r}({\bf k}))=\sum_{\mathcal{T}_\text{\rm v}({\bf w})}
E_{\mathcal{T}_\text{\rm r}({\bf k})}(\mathcal{T}({\bf a}),m,\mathcal{T}_\text{\rm v}({\bf w}))\sigma_{\bf x}(Y)\otimes\sigma_{{\bf a}-{\bf x}}(Y)
\end{equation}

where the sum over $\mathcal{T}_\text{\rm v}({\bf w})$ is as in Lemma \ref{PhiHahn} (with the same conditions on ${\bf x}=(x_1,\dotsc,x_h)$).
\end{lemma}

This Lemma gives immediately some of the basic properties of the $\Psi$ functions.

\begin{corollary}
The set 

\[\{\Psi(\mathcal{T}({\bf a}),m,\mathcal{T}_\text{\rm r}({\bf k})):\mathcal{T}_\text{\rm r}({\bf k})\;\text{\rm is a rep. labeling }\}
\]

is an orthogonal basis for the space of $S_{n-m}\times S_m$-invariant functions in $M^{\bf a}$.
\end{corollary}

\begin{corollary}
\[
\lVert \Psi(\mathcal{T}({\bf a}),m,\mathcal{T}_\text{\rm r}({\bf k}))\rVert^2_{M^{\bf a}}=\frac{a_1!a_2!\dotsb a_h!}{(n-m)!m!}\lVert \Phi(\mathcal{T}({\bf a}),m,\mathcal{T}_\text{\rm r}({\bf k}))\rVert^2_{M^{n-m,m}}.
\]
\end{corollary}

We need a little more work to establish the properties of the $\Psi$ functions corresponding to \eqref{Phifunctions} and Theorem \ref{belsub}.
If ${\bf x}=(x_1,\dotsc,x_h)$ is an $h$-parts composition and $1\leq t\leq h$ is fixed, as in section \ref{Labeled} we set ${\bf x}'=(x_1,x_2,\dotsc,x_t)$ and ${\bf x}''=(x_{t+1},x_{t+2},\dotsc,x_h)$. 
In the present setting, \eqref{LX} is: 

\begin{equation}%\label{IndMtensorM}
M^{\bf a}\equiv \text{Ind}^{S_n}_{S_{\underline{a}_t\times S_{n-\underline{a}_t}}}\left[M^{{\bf a}'}\otimes M^{{\bf a}''}\right]=\bigoplus_{F\in \Omega_{\underline{a}_t}}\left[M^{{\bf a}'}(F)\otimes M^{{\bf a}''}(F^C)\right];
\end{equation}

now $Y$ is replaced by $\Omega_{\underline{a}_t}$, $X$ by $\Omega_{{\bf a}}$ and $Z$ by $\Omega_{{\bf a}'}\times \Omega_{{\bf a}''}$. If we set

\[
\Theta_w=\{(B_1,\dotsc,B_h)\in\Omega_{\bf a}:\lvert (B_1\cup\dotsb\cup B_t)\cap Y\rvert =m-w\}
\]

then Mackey's Lemma \eqref{Mackey} in this setting is just

\begin{equation}\label{Mackey2}
\text{Res}^{S_n}_{S_{n-m}\times S_m}M^{\bf a}=\bigoplus_{w=\min\{0,m-\underline{a}_t\}}^{\min\{m,n-\underline{a}_t\}}L(\Theta_w)\equiv
\bigoplus_{w=\min\{0,m-\underline{a}_t\}}^{\min\{m,n-\underline{a}_t\}}
\bigoplus_{\substack{F\in\Omega_{\underline{a}_t}:\\ \lvert F\cap Y\rvert =m-w}}
\left[M^{{\bf a}'}(F)\otimes M^{{\bf a}''}(F^C)\right].
\end{equation}

Now we use Frobenius reciprocity to construct a sequence of irreducible representations from the sequence \eqref{Wchain}. The reciprocal of \eqref{Res} is

\[
M^{\bf a}=\text{Ind}^{K_0}_{K_1}\text{Ind}^{K_1}_{K_2}\dotsb\text{Ind}^{K_{L-1}}_{K_L}\left[S^{(a_1)}\otimes\dotsb\otimes S^{(a_h)}\right].
\]

Then we can define a sequence $V_0,V_1,\dotsc,V_L$ where  $V_j$ is an irreducible representations of $K_j$ isomorphic to $W_j$, $V_L$ is $S^{(a_1)}\otimes\dotsb\otimes S^{(a_h)}$ and, by backward induction  ($j=L-1,L-2,\dotsc,1,0$), $V_j$ is the unique subspace in $\text{Ind}^{K_j}_{K_{j+1}}V_{j+1}$ isomorphic to $W_j$ (it exists and is unique because the multiplicity of $W_{j+1}$ in $\text{Res}^{K_j}_{K_{j+1}}W_j$ is one).
Then $V_0$ is a subspace of $M^{\bf a}$ isomorphic to $S^{n-k,k}$, and we denote it by
$\left[S^{n-k,k}\right]_{\mathcal{T}_\text{r}({\bf k})}$. Clearly,

\begin{equation}\label{Trk}
\bigoplus_{\substack{\mathcal{T}_\text{r}({\bf k})\\k\text{ \rm fixed}}}\left[S^{n-k,k}\right]_{\mathcal{T}_\text{r}({\bf k})}
\end{equation}

is an orthogonal decomposition of the $S^{n-k,k}$-isotypic component in $M^{\bf a}$,  reciprocal to \eqref{dec}.
%Note also that $\left[S^{n-k,k}\right]_{\mathcal{T}_\text{r}({\bf k})}$ is the unique subspace of

%\begin{equation}\label{IndStensorS}
%\bigoplus_{F\in\Omega_{\underline{a}_t}}\left\{\left[S^{\underline{a}_t-i,i}(F)\right]_{\mathcal{T}'_\text{r}({\bf k})}\otimes\left[S^{n-\underline{a}_t-j,j}(F^C)\right]_{\mathcal{T}''_\text{r}({\bf k})} \right\}\cong
%\text{Ind}^{S_n}_{S_{\underline{a}_t\times S_{n-\underline{a}_t}}}\left[S^{\underline{a}_t-i,i}\otimes S^{n-\underline{a}_t-j,j}\right]
%\end{equation}

%isomorphic to $S^{n-k,k}$ (here $\left[S^{\underline{a}_t-i,i}(F)\right]_{\mathcal{T}'_\text{r}({\bf k})}$ and $\left[S^{n-\underline{a}_t-j,j}(F^C)\right]_{\mathcal{T}''_\text{r}({\bf k})}$ are the obvious subspaces of $M^{{\bf a}'}(F)$ and $ M^{{\bf a}''}(F^C)$ in \eqref{IndMtensorM}).

\begin{lemma}
If $v=\Phi(\mathcal{T}({\bf a}),k,\mathcal{T}_\text{\rm r}({\bf k}))$ and $T_v$ is the corresponding intertwining operator as in \eqref{proiezione}, then $T_vS^{n-k,k}=\left[S^{n-k,k}\right]_{\mathcal{T}_\text{r}({\bf k})}$.
\end{lemma}

\begin{proof}
Let $T_j:W_j\rightarrow\text{Ind}_{K_L}^{K_j}W_L$ be the intertwining operator associated to the $S_{\bf a}$-invariant vector $v$ (which belongs to each $W_j$ by Theorem \ref{belsub}); in particular, $T_0\equiv T_v$. Then a repeated application of Lemma \ref{traind} yields: if $T_jW_j=V_j$ then $T_{j-1}W_{j-1}\subseteq \text{Ind}_{K_j}^{K_{j-1}}V_j$, and therefore $T_{j-1}W_{j-1}=V_{j-1}$. This ends with $T_vS^{n-k,k}=V_0$.
\end{proof}

Then from Corollary \ref{etatilde} and Lemma \ref{LemmaPsifun} we get immediately

\begin{corollary}
The $S_{n-m}\times S_m$-invariant vector $\Psi(\mathcal{T}({\bf a}),m,\mathcal{T}_\text{\rm r}({\bf k}))$ belongs to $\left[S^{n-k,k}\right]_{\mathcal{T}_\text{r}({\bf k})}$. 
\end{corollary}

In the last part of this section, we establish
the inductive way to describe the $\Psi$ functions (reciprocal to \eqref{Phifunctions}). For any $F\in\Omega_{\underline{a}_t}$, with $\lvert F\cap Y\rvert=m-w$, we can take $\Psi(\mathcal{T}'({\bf a}),m-w,\mathcal{T}'_\text{\rm r}({\bf k}))\in M^{{\bf a}'}(F)$ and $\Psi(\mathcal{T}''({\bf a}),w,\mathcal{T}''_\text{\rm r}({\bf k}))\in M^{{\bf a}''}(F^C)$. We denote by 

\[
\left[\Psi(\mathcal{T}'({\bf a}),m-w,\mathcal{T}'_\text{\rm r}({\bf k}))\otimes \Psi(\mathcal{T}''({\bf a}),w,\mathcal{T}''_\text{\rm r}({\bf k}))\right]_F
\]

their tensor product and we set:

\[
\Xi(\mathcal{T}({\bf a}),w,\mathcal{T}_\text{\rm r}({\bf k}))=\sum_{F\in\Omega_{\underline{a}_t}:\\ \lvert F\cap Y\rvert=m-w}
\left[\Psi(\mathcal{T}'({\bf a}),m-w,\mathcal{T}'_\text{\rm r}({\bf k}))\otimes \Psi(\mathcal{T}''({\bf a}),w,\mathcal{T}''_\text{\rm r}({\bf k}))\right]_F.
\]

In the notation of \eqref{Mackey2}, we have $\Xi(\mathcal{T}({\bf a}),w,\mathcal{T}_\text{\rm r}({\bf k}))\in\Theta_w$.

\begin{proposition}
For the $\Psi$ function, we have the following iterative formula:

\begin{multline*}
\Psi(\mathcal{T}({\bf a}),m,\mathcal{T}_\text{\rm r}({\bf k}))=\sum_{w=\max\{j,i-\underline{a}_t+m\}}^{\min\{n-\underline{a}_t-j,m-i\}}E_{k-i-j}(n-\underline{a}_t-2j,\underline{a}_t-2i,m-i-j,w-j)
\times\\
\times\Xi(\mathcal{T}({\bf a}),w,\mathcal{T}_\text{\rm r}({\bf k})).
\end{multline*}
\end{proposition}

\begin{proof}
Keeping into account the decomposition \eqref{Mackey2}, this formula is just a rearrangement of \eqref{Psifunctions}.
\end{proof}

%%%%%%%%%%%%%%%%%%%%%%%%%%%%%%%%%%%%%%%%%%%%%%%%%%%%%%%%%%%%%%%%%%%%%%%
%%%%%%%%%%%%%%%%%%%%%%%%%%%%%%%%%%%%%%%%%%%%%%%%%%%%%%%%%%%%%%%%
\subsection{Regge's symmetries for the Hahn polynomials}\label{secRegge}
%%%%%%%%%%%%%%%%%%%%%%%%%%%%%%%%%%%%%%%%%%%%%%%%%%%%%%%%%%%%%%%%%%%%%%%
In this section, we want to show how the tools developed in the setting of the multidimensional Hahn polynomials may be also used to sketch a group theoretic proof/interpretation of the Regge symmetries for the one variable Hahn polynomials. 

\begin{lemma}\label{Regge}
The Hahn polynomials satisfy the following symmetry relations:

\begin{equation}\label{form5}
E_m(a,b,c,x)=(-1)^mE_m(b,a,c,c-x),
\end{equation}

\begin{equation}\label{formm7}
E_m(a,b,c,x)=(-1)^{a-x}\frac{m!(c-x)!(b-c+x)!}{(a+b-c-m)!(c-m)!(a-m)!}E_{a-m}(a,a+b-2m,a+b-m-c+x,x),
\end{equation}

\begin{equation}\label{form6}
E_m(a,b,c,x)=E_m(a+b-c,c,b,b-c+x).
\end{equation}

\end{lemma}
\begin{proof}
Consider the Hahn polynomials (and the associated intertwining functions) in the Example \ref{ah}. Suppose $h=1$ and take $2n$ in place of $n$. Let $(a_1,a_2,a_3)$ be a composition of $2n$ with $a_1,a_2,a_3\leq n$ and set $k_1=w_0=n$. These conditions force $k_2=n-a_1$. That is, $S^{(a_1)}\otimes S^{n,n-a_1}$ is the unique subrepresentation of $\text{Res}^{S_{2n}}_{S_{a_1}\times S_{a_2+a_3}}S^{n,n}$ of the form $S^{(a_1)}\otimes S^{a_2+a_3-j,j}$ (Corollary \ref{LR}). A consequence of this fact is that the dimension of the space of $S_{a_1}\times S_{a_2}\times S_{a_3}$-invariant vectors in $S^{n,n}$ is one (it is not zero because $n\geq a_1,a_2,a_3$). Moreover, from the results in section \ref{secPhiHahn}, we know that this space of invariant vectors is spanned by the function

\begin{equation}\label{Phi1}
\begin{split}
\Phi_1(a_1,a_2,a_3)=&\sum_{w=n-a_1}^n\sum_{u=\max\{0,w-a_2\}}^{\min\{a_3,w\}} (-1)^{n-w}a_1!(n-w)!(w-n+a_1)!\times\\
&\times E_{n-a_1}(a_3,a_2,w,u)\sigma_{n-w}(A_1)\otimes \sigma_{w-u}(A_2)\otimes \sigma_u(A_3).
\end{split}
\end{equation}

(we have also used the identity $E_{a_1}(a_1,a_1,a_1,w-n+a_1)=(-1)^{n-w}a_1!(n-w)!(w-n+a_1)!$; see \eqref{form9}). 
Clearly,
for every permutation $\pi$ of $\{1,2,3\}$, $\Phi_1(a_{\pi(1)},a_{\pi(2)},a_{\pi(3)})$ is still an $S_{a_1}\times S_{a_2}\times S_{a_3}$-invariant vector in $S^{n,n}$. Therefore there exists a complex number $\lambda_\pi$ such that $\Phi_1(a_1,a_2,a_3)=\lambda_\pi\Phi_1(a_{\pi(1)},a_{\pi(2)},a_{\pi(3)})$. Noting that

\begin{itemize}

\item
the coefficient of $\sigma_{a_1}(A_1)\otimes \sigma_{n-a_1}(A_2)\otimes \sigma_0(A_3)$ in \eqref{Phi1} is $(-1)^n(a_1!)^2(n-a_1)!(n-a_2+1)_{n-a_1}$,
\item
the coefficient of $\sigma_{a_1}(A_1)\otimes \sigma_0(A_2)\otimes \sigma_{n-a_1}(A_3)$ in \eqref{Phi1} is $(-1)^{a_1}(a_1!)^2(n-a_1)!(n-a_3+1)_{n-a_1}$,
\item
the coefficient of $\sigma_{n-a_2}(A_1)\otimes \sigma_{a_2}(A_2)\otimes \sigma_0(A_3)$ in \eqref{Phi1} is $(-1)^{a_1+a_2}a_1!a_2!a_3!$,
\end{itemize}

one can easily conclude that $\lambda_{(23)}=(-1)^{n-a_1}$ and $\lambda_{(12)}=(-1)^{n-a_3}\frac{a_1!(n-a_1)!}{a_2!(n-a_2)!}$. This yields \eqref{form5} and \eqref{formm7}. \\

To get \eqref{form6}, first note that if $\eta=\Phi(n,h,m,k)$ and $\eta_1=\Phi(n,m,h,k)$ then from Corollary \eqref{etatilde} we get

\[
\widetilde{\eta}(g)=\frac{(-1)^k}{(n-2k+2)_kk!}\langle \Phi(n,h,k,k),g\Phi(n,m,k,k)\rangle
\]

and

\[
\widetilde{\eta}_1(g)=\frac{(-1)^k}{(n-2k+2)_kk!}\langle \Phi(n,m,k,k),g\Phi(n,h,k,k)\rangle
\]

and therefore

\begin{equation}\label{symmetry}
\widetilde{\eta}(g^{-1})=\widetilde{\eta}_1(g).
\end{equation}

On the other hand, if $A,B$ are as in section \ref{secLR} and $Y,Z$ as in Lemma \ref{RmkSnk}, then

\begin{equation}\label{nmkknhkk}
\widetilde{\eta}(g)=E_k(n-h,h,m,\lvert gY\cap B\rvert)\qquad\quad\text{and}\qquad\quad\widetilde{\eta}_1(g)=E_k(n-m,m,h,\lvert gA\cap Z\rvert).
\end{equation}

Indeed, $\sigma_{m-w}(A)\otimes \sigma_w(B)$ is the characteristic function of the set $\{F\in\Omega_m:\lvert F\cap B\rvert=w\}$. Then \eqref{form6} follows from \eqref{symmetry} and \eqref{nmkknhkk}, noting that $\lvert g^{-1}Y\cap B\rvert=m-\lvert Y\cap gA\rvert=m-h+\lvert gA\cap Z\rvert$.
\end{proof}

To give a more symmetric formulation of Lemma \ref{Regge}, we introduce a Regge's array notation for the Hahn polynomials. Set

\[
\begin{split}
\boxed{
\begin{array}{ccc}
a_{11}&a_{12}&a_{13}\\
a_{21}&a_{22}&a_{23}\\
a_{31}&a_{32}&a_{33}
\end{array}}
=&(-1)^{a_{11}+a_{22}+a_{33}}\left[\frac{a_{31}!a_{32}!a_{13}!a_{23}!}{a_{22}!a_{21}!a_{11}!a_{12}!a_{33}!(a_{11}+a_{12}+a_{13})!}\right]^{1/2}\times\\
&\times E_{a_{33}}(a_{11}+a_{12},a_{21}+a_{22},a_{12}+a_{22},a_{12}).
\end{split}
\]

In the array, the sum of each row and the column has the same value $L$, and this fact determine $a_{13},a_{23},a_{31},a_{32},a_{33}$. 
The array is invariant under transposition and under even permutations of the rows (or of the columns), while is multiplied by $(-1)^L$ under an odd permutation of the rows (or of the columns). Indeed, the invariance under transposition is equivalent to \eqref{form6}, exchanging the first row with the second row we obtain \eqref{form5}, and exchanging the second row with the third row we obtain \eqref{formm7}.

\begin{remark}
{\rm
In Lemma \ref{Regge} we have used two different invariant vectors to get all the symmetry relations for the Hahn polynomials. We suppose that developing the techniques in \cite{ScarabottiSabc} (namely the theory of $S_{a_1}\times S_{a_2}\times S_{a_3}-S_{b_1}\times S_{b_2}\times S_{b_3}$-invariant vectors on $S_n$) one might get a single vector whose invariance properties yield all the Regge symmetries. This approach might be connected with \cite{Flamand}.

}
\end{remark}

%%%%%%%%%%%%%%%%%%%%%%%%%%%%%%%%%%%%%%%%%%%%%%%%%%%%%%%%%%%%%%%%%%%%%%%
%%%%%%%%%%%%%%%%%%%%%%%%%%%%%%%%%%%%%%%%%%%%%%%%%%%%%%%%%%%%%%%%
\section{Hahn polynomials and Clebsch-Gordan coefficients}
%%%%%%%%%%%%%%%%%%%%%%%%%%%%%%%%%%%%%%%%%%%%%%%%%%%%%%%%%%%%%%%%%%%%%%%
%%%%%%%%%%%%%%%%%%%%%%%%%%%%%%%%%%%%%%%%%%%%%%%%%%%%%%%%%%%%%%%%

%%%%%%%%%%%%%%%%%%%%%%%%%%%%%%%%%%%%%%%%%%%%%%%%%%%%%%%%%%%%%%%%%%%%%%%
%%%%%%%%%%%%%%%%%%%%%%%%%%%%%%%%%%%%%%%%%%%%%%%%%%%%%%%%%%%%%%%%
\subsection{Clebsch-Gordan coefficients}
%%%%%%%%%%%%%%%%%%%%%%%%%%%%%%%%%%%%%%%%%%%%%%%%%%%%%%%%%%%%%%%%%%%%%%%
%%%%%%%%%%%%%%%%%%%%%%%%%%%%%%%%%%%%%%%%%%%%%%%%%%%%%%%%%%%%%%%%

Now we recall some basic facts on the representation theory of $SU(2)$ \cite{BtD} and the theory of Clebsch-Gordan coefficients \cite{Koornwinder,K-V}.
For $l\in\frac{1}{2}\mathbb{Z}\equiv \left\{0,\frac{1}{2},1,\frac{3}{2},2,\dotsc\right\}$,
let $V_l$ be the space of all homogeneous polynomials of degree $2l$ in two variables $x,y$. We can define a representation $T^l$ of the group $SU(2)$ on $V_l$ by setting

\[
[T^l(g)P](x,y)=P(ax+cy,bx+dy)
\]

for $g= \bigl(\begin{smallmatrix}a&b\\c&d
\end{smallmatrix}\bigr)\in SU(2)$ and $P\in V_l$;
the set $\{T^l:l\in\frac{1}{2}\mathcal{Z}\}$ form a complete list of irreducible representations of $S(U(2)$. In $V_l$ we introduce the basis $\psi^l_j=\sqrt{\binom{2l}{l-j}} x^{l-j}y^{l+j}$, $\lvert j\rvert \leq l$ (where this means that $j=-l,-l+1,\dotsc,l$) and we suppose that $V_l$ is endowed with a scalar product in which this basis is orthonormal. This basis behaves nicely with respect to the restriction of $V_l$ to the subgroup $U(1)=\left\{\bigl(\begin{smallmatrix}a&0\\0&a^{-1}
\end{smallmatrix}\bigr):\lvert a\rvert =1\right\}$: if $g=\bigl(\begin{smallmatrix}a&0\\0&a^{-1}
\end{smallmatrix}\bigr)$ then $T^l(g)\psi_j^l=a^{-2j}\psi_j^l$, that is $\psi_j^l$ spans an $U(1)$-invariant one dimensional subspace of $V_l$ (corresponding to the character $a\mapsto a^{-2j}$) and for different values of $j$ we get different characters of $U(1)$. The tensor product of two irreducible representations of $SU(2)$ decomposes without multiplicity and according to the Clebsch-Gordan formula :

\begin{equation}\label{Clebsch}
V_{l_1}\otimes V_{l_2}=\bigoplus_{l=\lvert l_1-l_2\rvert}^{l_1+l_2} V_l.
\end{equation}

In the tensor product $V_{l_1}\otimes V_{l_2}$ we may introduce two orthonormal bases: the first is simply $\psi_{j_1}^{l_1}\otimes \psi_{j_2}^{l_2}$, $\lvert j_1\rvert\leq l_1$, $\lvert j_2\rvert\leq l_2$. But we can also introduce a basis that behaves nicely with respect to the Clebsch-Gordan decomposition
\eqref{Clebsch}: in each subspace $V_l$ in the right hand side of \eqref{Clebsch} we can choose a vector corresponding to $\psi^l_j$, and we denote this vector by $\phi_j^{l_1,l_2,l}$.
These vectors are uniquely determined by the requirement that they form an orthonormal basis plus the condition 

\begin{equation}\label{positivity}
\langle \phi_l^{l_1,l_2,l}, \psi_{l_1}^{l_1}\otimes \psi_{l-l_1}^{l_2}\rangle>0. 
\end{equation}

Indeed, this additional condition determine a unique unitary operator that intertwines $V_l$ and $V_{l_1}\otimes V_{l_2}$. The vectors of the second basis may be expressed in term of the vectors of the first basis:

\begin{equation}\label{CGcoef}
\phi_j^{l_1,l_2,l}=\sum_{j_1=-l_1}^{l_1}\sum_{j_2=-l_2}^{l_2}C_{j_1,j_2,j}^{l_1,l_2,l}\;
\psi_{j_1}^{l_1}\otimes \psi_{j_2}^{l_2}.
\end{equation}

The numbers $C_{j_1,j_2,j}^{l_1,l_2,l}$ are called {\em Clebsch-Gordan coefficients}. Note that we have $C^{l_1,l_2,l}_{j_1,j_2,j}=0$ if $j_1+j_2\neq j$.

%%%%%%%%%%%%%%%%%%%%%%%%%%%%%%%%%%%%%%%%%%%%%%%%%%%%%%%%%%%%%%%%%%%%%%%%%%%%%%%
%%%%%%%%%%%%%%%%%%%%%%%%%%%%%%%%%%%%%%%%%%%%%%%%%%%%%%%%%%%%%%%%%%%%%%%%%%%%%
\subsection{The reciprocity between the Gelfand pair $(S_n,S_{n-m}\times S_m)$ and the Clebsch-Gordan decomposition}\label{recgelcle}
%%%%%%%%%%%%%%%%%%%%%%%%%%%%%%%%%%%%%%%%%%%%%%%%%%%%%%%%%%%%%%%%%%%%%%%%%%%%%%%
%%%%%%%%%%%%%%%%%%%%%%%%%%%%%%%%%%%%%%%%%%%%%%%%%%%%%%%%%%%%%%%%%%%%%%%%%%%%%%%

The representation theory of the classical groups $GL(n,\mathbb{C}),U(n)$ and $SU(n)$ can be obtained from the representation theory of the symmetric group, using the so called Schur-Weyl reciprocity; see \cite{Ste}. In this paper, we use James version of this reciprocity \cite{Ja2}, section 26, which is quite useful for our purposes. 
Set $A=\{1,2,\dotsc,h\}$ and $B=\{h+1,h+2,\dotsc,n\}$, as in section \ref{secLR}.
For $h=0,1,\dotsc,n$ we define a linear map $\mathcal{X}_h:M^{n-m,m}\rightarrow V_{\frac{n-m}{2}}\otimes V_{\frac{m}{2}}$ by: if $X\in\Omega_m$ and $\lvert X\cap B\rvert=w$ 
then we set

\[
\mathcal{X}_h\delta_X=\frac{1}{\sqrt{\binom{n-m}{h-m+w}\binom{m}{w}}}\psi_{\frac{n+m}{2}-h-w}^{\frac{n-m}{2}}\otimes \psi_{w-\frac{m}{2}}^{\frac{m}{2}}\equiv 
x_1^{h-m+w}y_1^{n-w-h}\otimes x_2^{m-w}y_2^w.
\]

In other words, the exponents of $x_1,y_1,x_2,y_2$ are equal respectively to the cardinalities of the sets $X^C\cap A,X^C\cap B,X\cap A,X\cap B$. The following decomposition is obvious:

\[
V_{\frac{n-m}{2}}\otimes V_{\frac{m}{2}}=\bigoplus_{h=0}^n\text{Im}\mathcal{X}_h.
\]

Moreover, $\text{Im}\mathcal{X}_h$ is the subspace of $\text{Res}^{SU(2)}_{SU(1)}(V_{\frac{n-m}{2}}\otimes V_{\frac{m}{2}})$ corresponding to the character $a\mapsto a^{2h-n}$. Clearly,

\begin{equation}\label{Xhsigma}
\mathcal{X}_h\left[\sigma_{m-w}(A)\otimes \sigma_w(B)\right]=\binom{h}{m-w}\binom{n-h}{w}x_1^{h-m+w}y_1^{n-w-h}\otimes x_2^{m-w}y_2^w
\end{equation}

and the restriction of $\mathcal{X}_h$ to the space of $S_{n-h}\times S_h$-invariant vectors in $M^{n-m,m}$ (that is \eqref{Xhsigma}) is injective.

\begin{lemma}\label{Xhdfpartial}
For every $f\in M^{n-m,m}$ we have

\[
\mathcal{X}_hdf=(x_1\frac{\partial}{\partial x_2}+y_1\frac{\partial}{\partial y_2})\mathcal{X}_hf\qquad\quad\text{and}\quad\qquad \mathcal{X}_hd^*f=(x_2\frac{\partial}{\partial x_1}+y_2\frac{\partial}{\partial y_1})\mathcal{X}_hf.
\]
\end{lemma}

\begin{proof}
First of all, suppose that $f=\sigma_{m-w}(A)\otimes \sigma_w(B)$. Then we have

\[
\begin{split}
\mathcal{X}_hd[\sigma_{m-w}(A)\otimes \sigma_w(B)]=&\mathcal{X}_h[(h-m+w+1)\sigma_{m-w-1}(A)\otimes \sigma_w(B)\\
&+(n-h-w+1)\sigma_{m-w}(A)\otimes \sigma_{w-1}(B)]\\
=&(h-m+w+1)\binom{h}{m-w-1}\binom{n-h}{w}x_1^{h-m+w+1}y_1^{n-w-h}\otimes x_2^{m-w-1}y_2^w\\
&+(n-h-w+1)\binom{h}{m-w}\binom{n-h}{w-1}x_1^{h-m+w}y_1^{n-w-h+1}\otimes x_2^{m-w}y_2^{w-1}\\
=&(m-w)\binom{h}{m-w}\binom{n-h}{w}x_1^{h-m+w+1}y_1^{n-w-h}\otimes x_2^{m-w-1}y_2^w\\
&+w\binom{h}{m-w}\binom{n-h}{w}x_1^{h-m+w}y_1^{n-w-h+1}\otimes x_2^{m-w}y_2^{w-1}\\
=&(x_1\frac{\partial}{\partial x_2}+y_1\frac{\partial}{\partial y_2})\mathcal{X}_h[\sigma_{m-w}(A)\otimes \sigma_w(B)].
\end{split}
\]

Therefore we have proved the first identity when $f$ is $S_{n-h}\times S_h$-invariant.
Now let $P_m$ be the orthogonal projection from $M^{n-m,m}$ onto the subspace of $S_{n-h}\times S_h$-invariant functions.
Clearly, if $\lvert X\cap B\rvert=w$ then $P_m\delta_X=\frac{1}{\binom{h}{m-w}\binom{n-h}{w}}\sigma_{m-w}(A)\otimes\sigma_w(B)$.
From this fact it follows that $P_{m-1}d=dP_m$ and $\mathcal{X}_hP_m=\mathcal{X}_h$. Therefore, for any $f\in M^{n-m,m}$ we have 

\[
\mathcal{X}_hdf=\mathcal{X}_hP_{m-1}df=\mathcal{X}_hdP_mf=(x_1\frac{\partial}{\partial x_2}+y_1\frac{\partial}{\partial y_2})\mathcal{X}_hP_mf
=(x_1\frac{\partial}{\partial x_2}+y_1\frac{\partial}{\partial y_2})\mathcal{X}_hf.
\]

The proof for $d^*$ is the same.
\end{proof}

The proof of the following Lemma is easy.

\begin{lemma}\label{partialcommutes}
The operators $(x_1\frac{\partial}{\partial x_2}+y_1\frac{\partial}{\partial y_2}):V_{\frac{n-m}{2}}\otimes V_{\frac{m}{2}}\rightarrow V_{\frac{n-m+1}{2}}\otimes V_{\frac{m-1}{2}}$ and $(x_2\frac{\partial}{\partial x_1}+y_2\frac{\partial}{\partial y_1}):V_{\frac{n-m}{2}}\otimes V_{\frac{m}{2}}\rightarrow V_{\frac{n-m-1}{2}}\otimes V_{\frac{m+1}{2}}$ commute with the action of $SU(2)$.

\end{lemma}

%\begin{proof}
%If $g= \bigl(\begin{smallmatrix}a&b\\c&d
%\end{smallmatrix}\bigr)\in SU(2)$ and $\phi\in V_{\frac{n-m}{2}}$ then 
%\begin{multline*}
%(x_1\frac{\partial}{\partial x_2}+y_1\frac{\partial}{\partial y_2})\left[T^{\frac{n-m}{2}}(g)\otimes T^{\frac{m}{2}}(g)\right](\phi\otimes x_2^{m-w} y_2^w)\\
%=(x_1\frac{\partial}{\partial x_2}+y_1\frac{\partial}{\partial y_2})\left\{
%[T^{\frac{n-m}{2}}(g)\phi]\otimes (ax_2+cy_2)^{m-w}(bx_2+dy_2)^w\right\}\\
%=[x_1T^{\frac{n-m}{2}}(g)\phi]\otimes \left[(m-w)a(ax_2+cy_2)^{m-w-1}(bx_2+dy_2)^w+
%wb(ax_2+cy_2)^{m-w}(bx_2+dy_2)^{w-1}\right]\\
%+[y_1T^{\frac{n-m}{2}}(g)\phi]\otimes \left[(m-w)c(ax_2+cy_2)^{m-w-1}(bx_2+dy_2)^w+
%wd(ax_2+cy_2)^{m-w}(bx_2+dy_2)^{w-1}\right]\\
%=[(m-w)(ax_1+cy_1)T^{\frac{n-m}{2}}(g)\phi]\otimes (ax_2+cy_2)^{m-w-1}(bx_2+dy_2)^w\\
%+[w(bx_1+dy_1)T^{\frac{n-m}{2}}(g)\phi] \otimes (ax_2+cy_2)^{m-w}(bx_2+dy_2)^{w-1}\\
%=\left[T^{(n-m+1)/2}(g)\otimes T^{(m-1)/2}(g)\right]
%[(m-w)(x_1\phi)\otimes x_2^{m-w-1}y_2^w+w(y_1\phi)\otimes x_2^{m-w}y_2^{w-1}]\\
%=\left[T^{(n-m+1)/2}(g)\otimes T^{(m-1)/2}(g)\right](x_1\frac{\partial}{\partial x_2}+y_1\frac{\partial}{\partial y_2})(\phi\otimes x_2^{m-w} y_2^w).
%\end{multline*}
%The proof for $(x_2\frac{\partial}{\partial x_1}+y_2\frac{\partial}{\partial y_1})$ is the same.
%\end{proof}

Now we can write the reciprocal of \ref{corradtr2} in Theorem \ref{corradtr}.

\begin{corollary}\label{Clebsch2}
The Clebsch-Gordan decomposition \eqref{Clebsch} may be also written in the form
\[
V_{\frac{n-m}{2}}\otimes V_{\frac{m}{2}}=\bigoplus_{k=0}^{\min\{n-m,m\}}\left(x_2\frac{\partial}{\partial x_1}+y_2\frac{\partial}{\partial y_1}\right)^{m-k}\left[(V_{\frac{n-k}{2}}\otimes V_{\frac{k}{2}})\cap \text{\rm Ker}(x_1\frac{\partial}{\partial x_2}+y_1\frac{\partial}{\partial y_2})\right]
\]

where the space $(V_{\frac{n-k}{2}}\otimes V_{\frac{k}{2}})\cap \text{\rm Ker}(x_1\frac{\partial}{\partial x_2}+y_1\frac{\partial}{\partial y_2})$ is isomorphic to $V_{\frac{n}{2}-k}$ and coincides with $\bigoplus_{h=k}^{n-k}\mathcal{X}_hS^{n-k,k}$.
\end{corollary}

\begin{theorem}

\begin{multline}\label{ClebschHahn}
C^{\frac{n-m}{2},\frac{m}{2},\frac{n}{2}-k}_{\frac{n+m}{2}-h-w,w-\frac{m}{2},\frac{n}{2}-h}\\
=(-1)^k
\left[\frac{(h-k)!(n-h-k)!(m-k)!(n-m-k)!(n-2k+1)}{(m-w)!(h-m+w)!(n-h-w)!w!(n-k+1)!k!}\right]^{1/2}
E_k(n-h,h,m,w)
\end{multline}
\end{theorem}

\begin{proof}
First of all, from Corollary \ref{spherical} and the definition of $\mathcal{X}_h$, we get 

\begin{equation}\label{XhPhi}
\begin{split}
\mathcal{X}_h\Phi(n,h,m,k)=&\sum_{w=\max\{0,-h+m\}}^{\min\{n-h,m\}}E_k(n-h,h,m,w)\binom{h}{m-w}\binom{n-h}{w}\times\\
&\times\frac{1}{\sqrt{\binom{n-m}{h-m+w}\binom{m}{w}}}\psi_{\frac{n+m}{2}-h-w}^{\frac{n-m}{2}}\otimes \psi_{w-\frac{m}{2}}^{\frac{m}{2}}.
\end{split}
\end{equation}

Moreover, Corollary \ref{Clebsch2} implies that $\mathcal{X}_h\Phi(n,h,m,k)\in V_{\frac{n}{2}-k}\cap\text{Im}\mathcal{X}_h\subseteq V_{\frac{n-m}{2}}\otimes V_{\frac{m}{2}}$ and therefore, since $\text{Im}\mathcal{X}_h$ corresponds to the $SU(1)$ character $a\mapsto a^{2h-n}$, there exists $\lambda\in\mathbb{C}$ such that: $\mathcal{X}_h\Phi(n,h,m,k)=\lambda\phi^{\frac{n-m}{2},\frac{m}{2},\frac{n}{2}-k}_{\frac{n}{2}-h}$. To compute $\lambda$, first observe that from \eqref{XhPhi} and the orthogonality relations \eqref{orthrel} it follows that

\[
\begin{split}
\lambda^2\equiv\lVert \mathcal{X}_h\Phi(n,h,m,k)\lVert^2_{V_{(n-m)/2}\otimes V_{m/2}}=&\frac{(n-k+1)!(h-k)!(n-h-k)!k!}{(n-2k+1)(m-k+1)_k(n-m-k+1)_k}\times\\
&\times\left[\frac{(h-k+1)_k(n-h-k+1)_k}{(m-k)!(n-m-k)!}\right]^2
\end{split}
\]

The positivity condition \eqref{positivity} is satisfied if and only if the coefficient of $\psi_{\frac{n-m}{2}}^{\frac{n-m}{2}}\otimes \psi_{\frac{m}{2}-k}^{\frac{m}{2}}$ in $\frac{1}{\lambda}\mathcal{X}_k\Phi(n,k,m,k)$ is positive. From the symmetry relation \eqref{form6} and from \eqref{form9} we get that this coefficient is a positive multiple of

\[
\frac{1}{\lambda}E_k(n-k,k,m,m-k)=\frac{1}{\lambda}E_k(n-m,m,k,0)=\frac{(-1)^k}{\lambda}k!(n-m-k+1)_k
\]

and therefore we must have 

\begin{equation}\label{lambda}
\begin{split}
\mathcal{X}_h\Phi(n,h,m,k)=&(-1)^k\sqrt{\frac{(n-k+1)!(h-k)!(n-h-k)!k!}{(n-2k+1)(m-k+1)_k(n-m-k+1)_k}}\times\\
&\times\frac{(h-k+1)_k(n-h-k+1)_k}{(m-k)!(n-m-k)!}\phi^{\frac{n-m}{2},\frac{m}{2},\frac{n}{2}-k}_{\frac{n}{2}-h}.
\end{split}
\end{equation}

Then \eqref{ClebschHahn} follows from \eqref{XhPhi} and \eqref{lambda}.

\end{proof}

We may also write \eqref{ClebschHahn} in the following form:

\begin{multline}\label{ClebschHahn2}
E_m(a,b,c,x)\\=\left[\frac{x!(c-x)!(b-c+x)!(a-x)!(a+b-m+1)!m!}{(c-m)!(a+b-c-m)!(a-m)!(b-m)!(a+b-2m+1)}\right]^{1/2}C^{\frac{b}{2},\frac{a}{2},\frac{a+b}{2}-m}_{\frac{b}{2}-c+x,\frac{a}{2}-x,\frac{a+b}{2}-c};
\end{multline}

we have used the identity $E_k(n-h,h,m,w)=(-1)^kE_k(m,n-m,h,m-w)$, which follow from \eqref{form5} and \eqref{form6}, to transform the right hand side of \eqref{ClebschHahn}.

%%%%%%%%%%%%%%%%%%%%%%%%%%%%%%%%%%%%%%%%%%%%%%%%%%%%%%%%%%%%%%%%%%%%
%%%%%%%%%%%%%%%%%%%%%%%%%%%%%%%%%%%%%%%%%%%%%%%%%%%%%%%%%%%%%%%%%%%%
\subsection{The reciprocity between $M^\mathbf{a}$ and $V_\mathbf{a}$}
%%%%%%%%%%%%%%%%%%%%%%%%%%%%%%%%%%%%%%%%%%%%%%%%%%%%%%%%%%%%%%%%%%%%
%%%%%%%%%%%%%%%%%%%%%%%%%%%%%%%%%%%%%%%%%%%%%%%%%%%%%%%%%%%%%%%%%%%%
%Suppose that $\mathbf{a}=(a_1,a_2,\dotsc,a_h)\Vdash n$. 
%If $\mathbf{A}=(A_1,A_2,\dotsc,A_h)\in\Omega_{\mathbf{a}}$, $1\leq i,j\leq h$ and $x\in A_j$, we set $\mathbf{A}(j\stackrel{x}{\rightarrow}i)=(A_1,\dotsc,A_i\cup \{x\},\dotsc,A_j\setminus\{x\},\dotsc,A_h)$. Then we define a linear operator $d_{i,j}$ by setting

%\[
%d_{i,j}\delta_{\mathbf{A}}=\sum_{x\in A_j}\delta_{\mathbf{A}(j\stackrel{x}{\rightarrow}i)}.
%\]

%for all $\mathbf{A}\in\Omega_{\mathbf{a}}$. If $\lambda=(\lambda_1,\lambda_2,\dotsc,\lambda_h)$ is a partition of $n$ then the irreducible representation of $S_n$ associated to $\lambda$ is given by $M^\lambda\cap\left(\cap_{j=1}^{h-1}\text{Ker}d_{j,j+1}\right)$; see \cite{Ja2,Sc2,Ste}. 

Now set $Y=\{1,2,\dotsc,m\}$ and $Z=\{m+1,m+2,\dotsc,n\}$, as in the proof of Lemma \ref{RmkSnk}. Set $V_\mathbf{a}=V_{\frac{a_1}{2}}\otimes V_{\frac{a_2}{2}}\otimes \dotsb\otimes V_{\frac{a_h}{2}}$. Following the last paragraph of \cite{Ja2}, we define a linear operator $\mathcal{X}_m:M^\mathbf{a}\rightarrow V_\mathbf{a}$ by setting, if $\mathbf{A}\in\Omega_\mathbf{l}(Y)\times \Omega_{\mathbf{a}-\mathbf{l}}(Z)$, with  $\mathbf{l}=(l_1,l_2,\dotsc,l_h)\Vdash m$, $\mathbf{l}\leq \mathbf{a}$ (see \eqref{Omegaxa}) 

\[
\mathcal{X}_m\delta_\mathbf{A}=x_1^{l_1}y_1^{a_1-l_1}\otimes x_2^{l_2}y_2^{a_2-l_2}\otimes\dotsb\otimes x_h^{l_h}y_h^{a_h-l_h}.
\] 

In particular, 

\[
\mathcal{X}_m[\sigma_\mathbf{l}(Y)\otimes \sigma_{\mathbf{a}-\mathbf{l}}(Z)]=\binom{h}{l_1,\dotsc,l_h}\binom{n-h}{a_1-l_1,\dotsc,a_h-l_h}x_1^{l_1}y_1^{a_1-l_1}\otimes x_2^{l_2}y_2^{a_2-l_2}\otimes\dotsb\otimes x_h^{l_h}y_h^{a_h-l_h}.
\]

Clearly we have

\[
V_\mathbf{a} =\bigoplus_{m=0}^n\text{Im}\mathcal{X}_m
\]

and this is also the decomposition of $V_\mathbf{a}$ under the action of $SU(1)$: $\text{Im}\mathcal{X}_m$ corresponds to the character $a\mapsto a^{2m-n}$.

Now we sketch the construction reciprocal to \eqref{Trk}. We do not use the general results of James (that indeed might be used to study more general settings) and we base our considerations simply on the results in section \ref{recgelcle} and the tree method.  We define inductively a subspace $V(\mathcal{T}({\bf a}),\mathcal{T}_{\rm r}({\bf r}))\cong V_{\frac{n}{2}-k}$ of $V_{\bf a}$ as follows. Suppose that we have defined $V(\mathcal{T}'({\bf a}),\mathcal{T}'_{\rm r}({\bf r}))\cong V_{\frac{\underline{a}_t}{2}-i}$ and $V(\mathcal{T}''({\bf a}),\mathcal{T}''_{\rm r}({\bf r}))\cong V_{\frac{n-\underline{a}_t}{2}-j}$. Then $V(\mathcal{T}({\bf a}),\mathcal{T}_{\rm r}({\bf r}))$ is the subspace of $V(\mathcal{T}'({\bf a}),\mathcal{T}'_{\rm r}({\bf r}))\otimes V(\mathcal{T}''({\bf a}),\mathcal{T}''_{\rm r}({\bf r}))$ isomorphic to $V_{\frac{n}{2}-k}$. The basis of the induction (i.e. the case of a tree of height 1) is given by the Clebsch-Gordan decomposition. Clearly,

\[
V_{\bf a}=\bigoplus_{\mathcal{T}(\bf k)}V(\mathcal{T}({\bf a}),\mathcal{T}_{\rm r}({\bf r}))
\]

where the sum is over all representation labelings.
Then we define $\psi(\mathcal{T}({\bf a}),m,\mathcal{T}(\bf k))$ as the vector in $V(\mathcal{T}({\bf a}),\mathcal{T}_{\rm r}({\bf r}))$ corresponding to $\psi^{\frac{n}{2}-k}_{\frac{n}{2}-m}$. Since it is defined up to a multiplicative constant, we may suppose that it is given by a repeated application of \eqref{CGcoef}. That is, 

\[
\psi(\mathcal{T}({\bf a}),m,\mathcal{T}({\bf k}))=\sum^ {\min\{n-\underline{a}_t-j,m-i\}}_{\max\{j,i-\underline{a}_t+m\}}
C^{\frac{\underline{a}_t}{2}-i,\frac{n-\underline{a}_t}{2}-j,\frac{n}{2}-k}_{\frac{\underline{a}_t}{2}-m+w,\frac{n-\underline{a}_t}{2}-w,\frac{n}{2}-m}
\psi(\mathcal{T}'({\bf a}),m-w,\mathcal{T}'({\bf k}))\otimes \psi(\mathcal{T}''({\bf a}),w,\mathcal{T}''({\bf k})).
\]

Then an easy induction yields:

\begin{proposition}
\begin{enumerate}
\item
For every representation labeling $\mathcal{T}_{\rm r}(\bf k)$ we have:

\[
\bigoplus_{m=k}^{n-k}\mathcal{X}_m\left[S^{n-k,k}\right]_{\mathcal{T}(\bf k)}=V_{\mathcal{T}_{\rm r}(\bf k)}.
\]

\item
We have
$\mathcal{X}_m\Psi(\mathcal{T}({\bf a}),m,\mathcal{T}({\bf k}))=\lambda \psi(\mathcal{T}({\bf a}),m,\mathcal{T}(\bf k))$, where the constant $\lambda$
 can be recursively computed by:
 
\begin{multline*}
\lambda=\lambda'\lambda''\times\\
\times\left[\frac{(w-j)!(m-i-w)!(\underline{a}_t-m-i+w)!(n-\underline{a}_t-w-j)!(n-k-i-j+1)!(k-i-j)!}{(m-k)!(n-m-k)!(n-\underline{a}_t-k-j+i)!(\underline{a}_t-k-i+j)!(n-2k+1)}\right]^{1/2}
\end{multline*}

and $\lambda'$ and $\lambda''$ are the constants for $\mathcal{T}'$ and $\mathcal{T}''$.

\end{enumerate}
\end{proposition}

This Proposition shows that the tree method for the $S_{n-m}\times S_m$-intertwining functions in $M^{\bf a}$ is equivalent to the tree method for the Wigner's coefficients (or $3nj$-coefficients) for $SU(2)$. For the latter, we refer to \cite{BL,JLV,VdJ}.

\qquad\\
\qquad\\
\noindent
FABIO SCARABOTTI, Dipartimento MeMoMat, Universit\`a degli Studi di Roma ``La Sapienza'', via A. Scarpa 8, 00161 Roma (Italy)\\
{\it e-mail:} {\tt scarabot@dmmm.uniroma1.it}\\

\end{document}